\documentclass[11pt,a4paper]{article}
\usepackage{amsmath, amscd, amssymb, amsthm, latexsym}

\theoremstyle{plain}
\newtheorem{theorem}{Theorem}[section]
\newtheorem{proposition}[theorem]{Proposition}
\newtheorem{corollary}[theorem]{Corollary}
\newtheorem{lemma}[theorem]{Lemma}

\theoremstyle{definition}
\newtheorem{definition}[theorem]{Definition}

\newtheorem{note}[theorem]{Note}
\theoremstyle{remark}
\newtheorem{notation}[theorem]{\bf Notation}
\newtheorem{remark}[theorem]{\bf Remark}

\newcommand{\bbbn}{\mathbb{N}}

\newcommand{\bbbx}{\mathbb{X}}
\newcommand{\bbby}{\mathbb{Y}}\newcommand{\bbbz}{\mathbb{Z}}

\newcommand{\ttp}{{\tt p}}
\newcommand{\ttq}{{\tt q}}

\newcommand{\ttu}{{\tt u}}
\newcommand{\ttx}{{\tt x}}
\newcommand{\tty}{{\tt y}}\newcommand{\ttz}{{\tt z}}
\newcommand{\leqct}{\leq_{\rm ct}}
\newcommand{\geqct}{\geq_{\rm ct}}

\newcommand{\eqct}{=_{\rm ct}}

\newcommand{\infct}{<_{\rm ct}}
\newcommand{\supct}{>_{\rm ct}}
\newcommand{\often[2]}{\ {\rm OftLess}^{#1}_{#2}\ }
\newcommand{\lless[2]}{\ \rm \ll^{#1}_{#2}\ }
\newcommand{\ggreater[2]}{\ \rm \gg^{#1}_{#2}\ }

\newcommand{\couple}[2]{\mbox{$\langle #1,#2 \rangle$}}

\newcommand{\words}{\{0,1\}^*}

\newcommand{\WN}{{\words\to\bbbn}}

\newcommand{\XN}{{\bbbx\to\bbbn}}

\newcommand{\XY}{{\bbbx\to\bbby}}

\newcommand{\NXN}{{\bbbn\times\bbbx\to\bbbn}}
\newcommand{\NWN}{{\bbbn\times\words\to\bbbn}}

\newcommand{\sizesub}{\footnotesize}

\newcommand{\submax}{{\mbox{{\rm\sizesub max}}}}
\newcommand{\submin}{{\mbox{{\rm\sizesub min}}}}

\newcommand{\PR[1]}{{PR^{#1}}}

\newcommand{\maxpr[1]}{{Max^{#1}_{PR}}}

\newcommand{\minpr[1]}{{Min^{#1}_{PR}}}
\newcommand{\minprA[1]}{{Min^{#1}_{PR^A}}}
\newcommand{\maxprA[1]}{{Max^{#1}_{PR^A}}}

\newcommand{\pc[1]}{{PC^{#1}}}
\newcommand{\minpc[1]}{{Min^{#1}_{PC}}}
\newcommand{\maxpc[1]}{{Max^{#1}_{PC}}}
\newcommand{\kmax[1]}{{K^{#1}_\submax}}
\newcommand{\calkmax[1]}{{{\cal K}^{#1}_\submax}}
\newcommand{\kminbiblio}{{K_\submin}}
\newcommand{\kmaxbiblio}{{K_\submax}}

\newcommand{\kmin[1]}{{K^{#1}_\submin}}
\newcommand{\calkmin[1]}{{{\cal K}^{#1}_\submin}}

\newcommand{\ki}{{K^\infty}}
\newcommand{\hi}{{H^\infty}}

\newcommand{\calUmin[1]}{{{\cal U}^{#1}_\submin}}
\newcommand{\calUmax[1]}{{{\cal U}^{#1}_\submax}}

\newcommand{\sipi}{(\Sigma^0_1\wedge\Pi^0_1)}

\newcommand{\truc}{\exists^{\leq\phi_i}\sipi}

\newcommand{\sipiA}{(\Sigma^{0,A}_1\wedge\Pi^{0,A}_1)}

\newcommand{\trucA}{\exists^{\leq\phi^A_i}\sipiA}

\begin{document}
%
\title{Refinment of the ``up to a constant" ordering
\\ using contructive co-immunity and alike.
\\ Application to the $Min/Max$ hierarchy
\\ of Kolmogorov complexities}
%
\author{
{\small\sc Marie Ferbus-Zanda}\\
{\footnotesize LIAFA, Universit\'e Paris 7}\\
{\footnotesize 2, pl. Jussieu 75251 Paris Cedex 05}\\
{\footnotesize France}\\
{\footnotesize\tt ferbus@logique.jussieu.fr}
\and {\small\sc Serge Grigorieff}\\
{\footnotesize LIAFA, Universit\'e Paris 7}\\
{\footnotesize 2, pl. Jussieu 75251 Paris Cedex 05}\\
{\footnotesize France}\\
{\footnotesize\tt seg@liafa.jussieu.fr}
}
\date{\today}
\maketitle
{{\footnotesize \textnormal \tableofcontents}
%
\begin{abstract}\noindent
We introduce orderings $\lless[{\cal C},{\cal D}] {\cal F}$
between total functions $f,g:\bbbn\to\bbbn$
which refine the pointwise ``up to a constant" ordering $\leqct$
and also insure that $f(x)$ is often much less than $g(x)$.
With such $\lless[{\cal C},{\cal D}] {\cal F}$'s, we prove a strong
hierarchy theorem for Kolmogorov complexities obtained with jump
oracles and/or $Max$ or $Min$ of partial recursive functions.
\\
We introduce a notion of {\em second order} conditional Kolmogorov
complexity which yields a uniform bound for the
``up to a constant" comparisons involved in the hierarchy theorem.
\end{abstract}
%
%
%
\section{Introduction}
\label{s:intro}
%
%
\subsection{Comparing total functions $\bbbn\to\bbbn$}
\label{ss:comp}
%
\begin{notation}\label{not:start0}
Equality, inequality and strict inequality up to a constant
between total functions $I\to\bbbn$, where $I$ is any set,
are denoted as follows:
\begin{eqnarray*}
f\ \leqct\ g &\Leftrightarrow&
\exists c\in\bbbn\ \forall x\in I\ f(x)\leq g(x)+c
\\
f\ \eqct\ g &\Leftrightarrow &
f\leqct g\ \wedge\ g\leqct f\\
&\Leftrightarrow &
\exists c\in\bbbn\ \forall x\in I\ |f(x)-g(x)|\leq c
\\
f\ \infct\ g &\Leftrightarrow &
f \leqct g\ \wedge\ \neg(g \leqct f)\\
&\Leftrightarrow &
f \leqct g\ \wedge\ \forall c\in\bbbn\ \exists x\in I\ g(x)>f(x)+c
\end{eqnarray*}
\end{notation}
Total functions $f,g:\bbbn\to\bbbn$ can be compared in diverse
ways. The simplest one is pointwise comparison via the partial
ordering relation $\forall x\ f(x)<g(x)$.
In case functions are considered up to an additive constant,
for instance with Kolmogorov complexity, pointwise comparison
has to be replaced by the $\leqct$ preordering or the $\infct$
ordering.
\\
Observe that the $\infct$ ordering is an infinite intersection:
$$f\infct g\ \Leftrightarrow\ f\leqct g\ \wedge\
\forall c\in\bbbn\ f<_{io}g-c$$
where $<_{io}$ (io stands for ``infinitely often") is
the non transitive relation
\\\centerline{$f<_{io}g\ \Leftrightarrow\
\{x:f(x)<g(x)\}\mbox{ is infinite}$}
Relation $<_{io}$ can be much refined via localization: instead of
merely demanding $\{x:f(x)<g(x)\}$ to be infinite, one can ask it
to have infinite intersection with every infinite set in a family
${\cal C}$ of sets.
\\
In case ${\cal C}$ is the family of all subsets of $\bbbn$,
this gives the relation
\\\centerline{$\{x:f(x)<g(x)\}\mbox{ is cofinite}$}
which is a partial ordering relation.
\\
In case ${\cal C}$ is the family of r.e. sets, this is
related to the idea of coimmunity.
\\
An instance of such a relation appears in a classical result about
Kolmogorov complexity $K$, due to Barzdins
(cf. \cite{livitanyi} Thm.2.7.1 iii, p.167, or
Zvonkin \& Levin, \cite{zvonkinlevin} p.92.),
which states that, for any total recursive function $\phi$ which tends
to $+\infty$, the set $\{x : K(x)<\phi(x)\}$ meets every infinite r.e.
set.
\medskip\\
In practice, for simple classes ${\cal C}$, an infinite subset
of $X\cap\{x:f(x)<g(x)\}$, for $X$ infinite in ${\cal C}$, can always
be found in a not too complex class ${\cal D}$.
Which leads to consider the relation $\often[{\cal C},{\cal D}]\,$
such that
\begin{eqnarray*}
f\often[{\cal C},{\cal D}]\,g&\Leftrightarrow&
\forall X\in{\cal C}\ \exists Y\in{\cal D}\
(X\mbox{ is infinite }\Rightarrow\\
&&\hspace{1cm}Y\mbox{ is infinite }
\wedge\ Y\subseteq\{x:f(x)<g(x)\})
\end{eqnarray*}
If ${\cal C}={\cal D}$ then this relation is transitive, hence is
a strict partial ordering.
However, in case ${\cal C}\neq{\cal D}$, transitivity may fail
(for instance, a counterexample is obtained via Lemma
\ref{l:ggKminKmax}).
\medskip\\
The key observation for the paper is as follows:
\medskip\\
{\em For any ${\cal C},{\cal D}$, the relation
$\ f\leqct g\ \wedge\ \forall c\ (f\often[{\cal C},{\cal D}]\,g-c)\ $
is transitive, hence is a partial strict ordering refining $\infct$.
In other words, considering $\often[{\cal C},{\cal D}]\,$ up to any
constant and mixing it with $\leqct$ always leads to an ordering.}
\medskip\\
If ${\cal F}$ is a family of total functions $\bbbn\to\bbbn$ which
tend to $+\infty$ and ${\cal F}$ is closed by translations
(i.e. $\phi\in{\cal F}$ implies $\max(0,\phi-c)\in{\cal F}$),
then the above observation also applies to the relation
$f\leqct g\ \wedge\
\forall\phi\in{\cal F}\ f\often[{\cal C},{\cal D}]\,\phi\circ g$,
i.e. the relation
\medskip\\\indent$f\leqct g\ \wedge\
\forall\phi\in{\cal F}\
\forall X\in{\cal C}\ \exists Y\in{\cal D}$

\hfill{$(X\mbox{ is infinite }\Rightarrow\ Y\mbox{ is infinite }
\wedge\ Y\subseteq\{x:f(x)<g(x)\})$}
\\
which is also a partial strict ordering refining the ordering
$\infct$.
\medskip\\
Enriching this relation with the requirement that a code for an
infinite subset $Y$ of $X\cap\{x:f(x)<\phi(g(x))\}$ can be
effectively computed from codes for $\phi$ and $X$, we get
the relation $\often[{\cal C},{\cal D}]{\cal F}$ which is the main
concern of this paper.
\medskip\\
In \S\ref{s:functional} we review some needed elements of oracular
computability. This is done in terms of partial computable
functionals so as to get uniformity in the oracle.
\medskip\\
In \S\ref{s:constructive} we recall Xiang Li's notion of
constructive immunity and introduce the related notions of
$({\cal C},{\cal D})$-density and constructive density.
\medskip\\
In \S\ref{s:ll} we introduce the relation
$\often[{\cal C},{\cal D}]{\cal F}$ and its variant
$\often[{\cal C},{\cal D}]{{\cal F}\uparrow}$ (where only total
monotone increasing functions in ${\cal F}$ are considered)
and prove that their intersections
$\lless[{\cal C},{\cal D}]{\cal F}$ and
$\lless[{\cal C},{\cal D}]{{\cal F}\uparrow}$
with $\leqct$ are strict orderings refining the ordering
$\infct$.
%
\subsection{Second order Kolmogorov complexity}
\label{ss:functionalK}
%
In relation with the partial computable functional approach to
oracular computability (cf. \S\ref{s:functional}), we develop in
\S\ref{s:uniform} a
{\em functional version ${\cal K}(x\,||\,A)$ of Kolmogorov
complexity.}
This amounts to a simple, seemingly unnoticed, fact:
\\
{\em Oracular Kolmogorov complexity $K^A$ can be obtained by
instantiating to $A$ the second order parameter of a variant of
conditional Kolmogorov complexity in which the condition is a
set of integers rather than an integer. The oracle is thus viewed as
a second order conditional parameter.}
\\
The usual proof of the invariance theorem goes through.
This second-order conditional complexity allows for a
uniform choice of oracular Kolmogorov complexities
(this is detailed in \S\ref{s:unifVersusOracle}) since,
for any $A$,
$${\cal K}(\ttx\,||\,A)\eqct K^A(\tt x)$$
i.e.
$\forall A\ \exists c\ \forall \ttx\
|K^A(\ttx)-{\cal K}(\ttx\,||\,A)|\leq c$.
\medskip\\
A typical benefit of the functional version of $K$ is as follows.
Usual properties with $K$ involving equality or inequality
``up to a constant" go through oracles. Let $c_A$ be the
involved constant for the oracle $A$ version.
For a single equality or inequality involving $K^A$,
it may be possible to modify $K^A$ (by an additive constant)
so that $c_A=0$. But this is no more possible for several equalities
or inequalities since the needed modifications of $K^A$ may
-- a priori -- be incompatible.
\\
Thus, for a system of equalities or inequalities, there is no a
priori $A$-computable bound of the involved constant $c_A$ for the
oracle $A$ version.
However, in case (which is also usual) such properties also
go through the functional version, the constant bound involved
in the functional version is valid for any oracle.
In other words, {\em whereas the oracular version a priori allows
no $A$-computable bound of the constant, the functional
version does allow a constant bound.}
\\
This fact is applied in \S\ref{ss:hierarchy} to get sharper
results.
\medskip\\
In \S\ref{s:KminKmax} we recall the variants $\kmax[],\kmin[]$ of
Kolmogorov complexity introduced in our paper
\cite{ferbusgrigoKmaxKmin} and we extend them to functional
versions. The precise relation between such functional versions
and the oracular $\kmax[],\kmin[]$ is detailed in
\S\ref{s:unifVersusOracle}.
%
%
\subsection{A strong hierarchy theorem for Kolmogorov complexities}
%
In \S\ref{ss:barzdins} we prove of a version of Barzdins'
result cited in \S\ref{ss:comp} (cf. also \S\ref{ss:immune})
with as much effectivity as possible which
involves an ordering relation introduced in \S\ref{s:ll} and
can be stated as $K\lless[\Sigma^0_1,\Sigma^0_1]{PR}\log$.
Also, the functional versions of Kolmogorov complexity
and the functional approach to oracular computability allow
to get a functional version of this result, hence to get
effectivity relative to the oracle.
\medskip\\
We extend this result in \S\ref{ss:KKmaxalabarzdins},
\ref{ss:KKminalabarzdins}, \ref{ss:KminKmaxOften}
and prove that $K,\kmax[],\kmin[]$ can be compared via the above
$\often[]\,$ and $\lless[]\,$ relations, with more complex
classes ${\cal C},{\cal D}$, namely
${\cal C}=\Sigma^0_1\cup\Pi^0_1$
and ${\cal D}=\exists^{<\phi}(\Sigma^0_1\wedge\Pi^0_1)$
or the variants in which $\Pi^0_1$ is constrained with a
``recursively bounded growth" condition
(cf. Def.\ref{def:recbounded}).
Also, the class ${\cal F}$ can be extended to $\minpr[]$, i.e.
the class of infima of partial recursive sequences of functions.
\medskip\\
The above class ${\cal D}$ is a subclass of $\Delta^0_2$ which can
be obtained via bounded existential quantification over boolean
combinations of $\Sigma^0_1$ relations.
In \S\ref{ss:complexity}, we show that such syntactical complexities
naturally appear when comparing $K,\kmax[],\kmin[]$.
\medskip\\
Finally, in \S\ref{ss:hierarchy} we prove the main application of
the $\lless[{\cal C},{\cal D}]{\cal F}$ and
$\lless[{\cal C},{\cal D}]{{\cal F}\uparrow}$ orderings, which is a
strong hierarchy theorem for the Kolmogorov complexities
$K,\kmax[],\kmin[]$ and their oracular versions using the
successive jumps.
%
%
\section{Partial computable functionals
and oracular recursion theory}
\label{s:functional}
%
%
\subsection{Notations}
\label{ss:notations}
%
\begin{notation}\label{not:start}$\\ $
{\bf 1.}
{\em [Basic sets]\ }
$\bbbx,\bbby$ denote products of
non empty finite families of spaces of the form
$\bbbn$ or $\bbbz$ or $\Sigma^*$ where $\Sigma$ is some finite
alphabet.
\medskip\\
{\bf 2.}{\em [Partial recursive functions]\ }
Let $A\subseteq\bbbn$.
We denote $PR^{\bbbx\to\bbby}$ (resp. $PR^{\bbbx\to\bbby,A}$)
the family of partial recursive (resp. $A$-recursive) functions
between basic sets $\bbbx$ and $\bbby$.
\medskip\\
{\bf 3.}{\em [Bijections between basic spaces]\ }
For any basic spaces $\bbbx,\bbby$ and $\bbbz$ we fix some
particular total recursive bijection from $\bbbx\times\bbby$ to
$\bbbz$ and denote ${\couple \ttx \tty}_{\bbbx\times\bbby,\bbbz}$,
or simply $\couple \ttx \tty$, the image in $\bbbz$ of the pair
$(\ttx,\tty)$.
\end{notation}
%
\subsection{Some classical results from recursion theory}
\label{ss:classicRec}
%
We shall use the following classical results from
computability theory
(cf. Odifreddi's book \cite{odifreddi} p.372--374, 288--292,
or Shoenfield's book \cite{shoenfield2}).
\begin{proposition}\label{p:classicRec}$\\ $
{\bf 1.} {\em (Post's Theorem, 1948 \cite{post48})}
A set is $\Sigma^0_{n+1}$ (resp. $\Delta^0_{n+1}$)
if and only if
it is recursively enumerable (r.e.) (resp. recursive)
in oracle $\emptyset^{(n)}$.
\medskip\\
{\bf 2.} {\em (Post, 1944 \cite{post44})}
For any oracle $A$, every infinite $A$-r.e. set $X$ contains
an infinite set $Y$ which is recursive in $A$.
Moreover, one can recursively go from an r.e. code
for $X$ to r.e. codes for such a $Y$ and its complement.
\medskip\\
In particular, every infinite $\Sigma^0_{n+1}$ set $X$
contains an infinite $\Delta^0_{n+1}$ subset $Y$.
Also, one can recursively go from a $\Sigma^0_{n+1}$-code
for $X$ to $\Sigma^0_{n+1}$-codes for such a $Y$
and its complement.
\medskip\\
{\bf 3.}
Recall that an $A$-r.e. set $X\subset\bbbn$ is maximal if
it is coinfinite and for any $A$-r.e. set $Y\supseteq X$
either $\bbbn\setminus Y$ is finite or $Y\setminus X$ is finite.
\\ {\em (Friedberg, 1958 \cite{friedberg58})}
There exists maximal $A$-r.e. sets.
\end{proposition}
\begin{remark}\label{rk:classicRec}$\\ $
{\bf 1.}
Since every $\Pi^0_n$ set is $\Sigma^0_{n+1}$, point 2
of the above proposition yields that every infinite $\Pi^0_n$
set contains an infinite $\Delta^0_{n+1}$ subset.
This cannot be improved: the complement of any maximal
recursively enumerable set is an infinite $\Pi^0_1$ set which
does not contain any infinite recursive set.
\medskip\\
{\bf 2.}
Any total function $\psi$ with graph in $\Sigma^0_n$ is
in fact $\emptyset^{(n-1)}$-recursive and has graph in
$\Delta^0_n$ since
$\ y\neq\psi(x)\ \Leftrightarrow\ \exists z\neq y\ z=\psi(x)$.
\end{remark}
%
\subsection{Partial computable functionals}
\label{ss:functionals}
%
Def.\ref{def:functionals} is classical, cf. Rogers
\cite{rogers} p.361, or Odifreddi \cite{odifreddi} p.178.
\begin{definition}\label{def:functionals}
A (partial) functional ${\cal F}:\bbbx\times P(\bbbn)\to\bbby$
is partial computable if there exists an oracle Turing machine
${\cal M}$ such that, given $A\in P(\bbbn)$ as oracle and
$\ttx\in\bbbx$ as input,
\\\indent- ${\cal M}$ halts and accepts if and only if
            ${\cal F}(A,\ttx)$ is defined,
\\\indent- if ${\cal M}$ halts and accepts then its output is
            ${\cal F}(A,\ttx)$.
\medskip\\
The family of partial computable functionals
$\bbbx\times P(\bbbn)\to\bbby$
is denoted $PC^{\bbbx\times P(\bbbn)\to\bbby}$.
\end{definition}
The notion of acceptable enumeration of partial recursive
functions (cf. Rogers \cite {rogers} Ex. 2.10 p.41,
or Odifrreddi \cite{odifreddi}, p.215)
extends to functionals.
\begin{definition}\label{def:acceptable}
We denote $\bbbx,\bbby,\bbbz$ some basic sets
(cf. Notation \ref{not:start}).
\medskip\\
{\bf 1.}
An enumeration $(\Phi_i)_{i\in\bbbn}$ of partial computable
functionals $\bbbx\times P(\bbbn)\to\bbby$
is {\em acceptable} if
\begin{enumerate}
\item[i.]
$(i,\ttx,A)\mapsto\Phi_i(\ttx,A)$ is a partial computable
functional.
\item[ii.]
Every partial computable functional $\bbbx\times P(\bbbn)\to\bbby$
is enumerated:
$$\forall \Psi\in PC^{\bbbx\times P(\bbbn)\to\bbby}\
\exists i\ \Phi_i=\Psi$$
\item[iii.]
the parametrization (also called s-m-n) property holds:
for every basic set $\bbbz$, there exists a total recursive
function $s^\bbbz_\bbbx:\bbbn\times\bbbz\to\bbbn$ such that
$$\forall i\ \forall\ttz\in\bbbz\ \forall\ttx\in\bbbx\
\forall A\subseteq\bbbn\ \
\Phi_i(\couple\ttz \ttx,A)
=\Phi_{s^\bbbz_\bbbx(i,\ttz)}(\ttx,A)$$
where $\couple\ttz\ttx$ is the image of the pair $(\ttz,\ttx)$
by some fixed total recursive bijection
$\bbbz\times\bbbx\to\bbbx$ (cf. Notation \ref{not:start}).
\end{enumerate}
{\bf 2.}
An enumeration $({\cal W}_i)_{i\in\bbbn}$ of
$\Sigma^0_1$ subsets of $\bbbx\times P(\bbbn)$
is {\em acceptable} if there exists an {\em acceptable}
enumeration $(\Phi_i)_{i\in\bbbn}$ of partial recursive
functionals such that ${\cal W}_i$ is the domain of $\Phi_i$.
\\
In particular, $({\cal W}_i)_{i\in\bbbn}$ is $\Sigma^0_1$ as
a subset of $\bbbn\times\bbbx\times P(\bbbn)$.
\end{definition}
\begin{proposition}\label{p:acceptable}
There exists an acceptable enumeration of partial computable
functionals $\bbbx\times P(\bbbn)\to\bbby$.
\end{proposition}
%
%
\subsection{Uniform relativization}
\label{ss:uniform}
%
When dealing with oracles $A$, it is often possible to get
results involving {\em recursive} transfer functions
rather than $A$-recursive ones.
To do so, we must consider enumerations of $A$-r.e. sets
and partial $A$-recursive functions which are obtained from
enumerations of {\em partial computable functionals} by fixing
the second order argument $A$.
Such enumerations will be called {\em uniform enumerations}.
\\
This amounts to consider relative computability as a concept
dependent on the prior notion of partial computable functional,
though, historically, relative computability came first,
cf. Hinman's book \cite{hinman} 5.15 p.68.
\begin{proposition}\label{p:uniform}$\\ $
Let $(\Phi_i)_{i\in\bbbn}$ be an acceptable enumeration of
partial computable functionals $\bbbx\times P(\bbbn)\to\bbby$.
\\
For $A\subseteq\bbbn$, define $\varphi^A_i:\bbbx\to\bbby$
and $W^A_i\subseteq\bbbx$ from $\Phi_i$ and ${\cal W}_i$ by
fixing the second order argument as follows:
\medskip\\\medskip\centerline{$\begin{array}{rclcrclcl}
\varphi^A_i(\ttx)&=&\Phi_i(\ttx,A)
&&W^A_i&=&domain(\varphi^A_i)&=&\{\ttx:(\ttx,A)\in{\cal W}_i\}
\end{array}$}
Then the sequences $(\varphi^A_i)_{i\in\bbbn}$ and
$(W^A_i)_{i\in\bbbn}$ are acceptable enumerations of the family
$PR^{\XY,A}$) of partial $A$-recursive functions $\XY$ and that of
$A$-r.e. subsets of $\bbbx$.
\\
Such acceptable enumerations are called {\em uniform enumerations}.
\end{proposition}
Rogers' theorem (cf. Odifreddi \cite{odifreddi} p.219) extends
to partial computable funtionals, hence to uniform enumerations.
\begin{theorem}\label{thm:rogers}$\\ $
{\bf 1. (Rogers' theorem)}
If $(\Psi_i)_{i\in\bbbn}$ and $(\Phi_i )_{i\in\bbbn}$ are both
acceptable enumerations of partial computable functionals
$\bbbx\times P(\bbbn)\to\bbby$ , then there exists some
recursive bijection $\theta:\bbbn\to\bbbn$ such that
$\Psi_i=\Phi_{\theta(i)}$ for all $i\in\bbbn$.
\medskip\\
{\bf 2.}
If $(\psi^A_i)_{i\in\bbbn}$ and $(\varphi^A_i)_{i\in\bbbn}$ are
uniform enumerations of partial $A$-recursive functions
then there exists some {\em recursive} bijection
$\theta:\bbbn\to\bbbn$ such that
$\psi^A_i=\varphi^A_{\theta(i)}$ for all $i\in\bbbn$.
\end{theorem}
Uniform enumerations allow for {\em effective} (as opposed to
$A$-effective) closure results for a lot of operations on
partial $A$-recursive functions and $A$-r.e. sets which
correspond to closure properties of partial computable
functionals admitting sets and partial functions as arguments,
cf. Hinman \cite{hinman} {\S}II.2, II.4.
%
\subsection{Acceptable enumerations of some subclasses of
$\Delta^0_2$}
\label{ss:Delta02}
%
Comparison of $K$ and $\kmin[],\kmax[]$ in the hierarchy theorem
\ref{thm:hierarchy} involves particular $\Delta^0_2$ sets described
in Def.\ref{def:Dclass} below.
First, we fix a notion of bounded quantification pertinent for
our applications.
\begin{definition}\label{def:existsleq}$\\ $
{\bf 1.}
We consider on each basic set a norm such that
\\\indent- $||x||=|x|$ if $x\in\bbbn\mbox{ or }\bbbz$,
\\\indent- $||x||=length(x)$ if $x\in\Sigma^*$ where $\Sigma$
            is a finite alphabet,
\\\indent- $||(x_1,...,x_k)||=\max(||x_1||,...,||x_1||)$.
\medskip\\
{\bf 2.}
Suppose $\mu:\bbbn\to\bbbn$ is a total function
(resp. $\mu:\bbbn\times P(\bbbn)\to\bbbn$ is a total functional)
which is monotone increasing (resp. with respect to its first
argument).
Let $\bbbx$ is a basic set.
\\
For $R\subseteq\bbbx\times(\words)^m$ and
${\cal R}\subseteq\bbbx\times(\words)^m\times P(\bbbn)$,
we let
\begin{eqnarray*}
\exists^{\leq\mu}R&=&\{\ttx : \exists\vec{\ttu}\
(|\ttu_1|,\ldots,|\ttu_m|\leq\mu(||\ttx||)\ \wedge\
R(\vec{\ttu},\ttx))\}
\\
\exists^{\leq\mu}{\cal R}&=&\{(\ttx,A) : \exists\vec{\ttu}\
(|\ttu_1|,\ldots,|\ttu_m|\leq\mu(||\ttx||)\ \wedge\
{\cal R}(\vec{\ttu},\ttx,A))\}
\end{eqnarray*}
If ${\cal C}\subseteq P(\bbbx)$
(resp. ${\cal C}\subseteq P(\bbbx\times P(\bbbn))$, we denote
$\ \exists^{\leq\mu}{\cal C}\ $
the subclass of subsets of $\bbbx$
(resp. $\bbbx\times P(\bbbn)$) consisting of all sets
$\ \exists^{\leq\mu}R\ $ where $R$ is in ${\cal C}$.
\end{definition}
\begin{note}
In view of applications to Kolmogorov complexity, we choose
bounded quantifications over binary words (where the bound applies
to the length).
Of course, going from $\mu$ to $2^{\mu}$, we can reduce to
bounded quantifications over $\bbbn$.
\end{note}
As is well known, bounded quantification does not increase
syntactical complexity of $\Delta^0_2$ sets.
\begin{proposition}\label{p:boundedDelta02}
If $\mu:\bbbn\to\bbbn$ has $\Sigma^0_2$ graph then
$\ \exists^{\leq\mu}\Delta^0_2 \subseteq\Delta^0_2$,
be it for relations in $\bbbx$ or in $\bbbx\times P(\bbbn)$.
\end{proposition}
\begin{proof}
In case $\mu(x)=x$ this is just the commutation of a bounded
quantification with an unbounded one.
In general, we have
\medskip\\
$\exists\vec{\ttu}\ (\vec{\ttu}\leq\mu(||\ttx||)\
\wedge\ R(\ttx,\vec{\ttu}))$
\\\indent$\Leftrightarrow\ \exists y\ (y=\mu(||\ttx||)\ \wedge\
\exists\vec{\ttu}\ (\vec{\ttu}\leq y\ \wedge\ R(\ttx,\vec{\ttu})))$
\\\indent$\Leftrightarrow\ \forall y\ (y=\mu(||\ttx||)\ \Rightarrow\
\exists\vec{\ttu}\ (\vec{\ttu}\leq y\ \wedge\ R(\ttx,\vec{\ttu})))$
\medskip\\
which are respectively $\exists\ (\Sigma^0_2\wedge\Delta^0_2)$,
hence $\Sigma^0_2$, and $\forall\ (\Pi^0_2\vee\Delta^0_2)$,
hence $\Pi^0_2$.
\end{proof}
\begin{definition}\label{def:Dclass}$\\ $
Let ${\cal C}$ be a syntactical class among
$$\Sigma^0_1\ \ ,\ \ \Pi^0_1\ \ ,\ \ \Sigma^0_1\vee\Pi^0_1
\ \ ,\ \ \exists^{\leq\mu}(\Sigma^0_1\wedge\Pi^0_1)$$
{\bf 1.}
Let ${\cal C}[\bbbx]$ be the family of subsets of $\bbbx$
which are ${\cal C}$-definable.
An acceptable enumeration $(W^{{\cal C}[\bbbx]}_i)_{i\in\bbbn}$
of ${\cal C}[\bbbx]$ is an enumeration obtained from acceptable
enumerations $(W^{\bbbx\times(\words)^m}_i)_{i\in\bbbn}$
of r.e. subsets of the $\bbbx\times(\words)^m$'s as follows:
\medskip\\\medskip\centerline{
$\begin{array}{rcll}
W^{\Pi^0_1[\bbbx]}_i&=&\bbbx\setminus W^\bbbx_i&
\\
W^{\Sigma^0_1\wedge\Pi^0_1[\bbbx]}_i
&=&W^\bbbx_j\cap(\bbbx\setminus W^\bbbx_k)&
\mbox{where }i=\couple j k
\\
W^{\Sigma^0_1\vee\Pi^0_1[\bbbx]}_i
&=&W^\bbbx_j\cup(\bbbx\setminus W^\bbbx_k)&
\mbox{where }i=\couple j k
\\
W^{\exists^{\leq\mu}(\Sigma^0_1\wedge\Pi^0_1)[\bbbx]}_i&=&
\exists^{\leq\mu}
W^{(\Sigma^0_1\wedge\Pi^0_1)[\bbbx\times(\words)^m]}_j&
\mbox{where }i=\couple j m
\end{array}$}
{\bf 2.}
Let ${\cal C}[\bbbx\times P(\bbbn)]$ be the family of subsets of
$\bbbx\times P(\bbbn)$ which are ${\cal C}$-definable.
An acceptable enumeration
$({\cal W}^{{\cal C}[\bbbx]}_i)_{i\in\bbbn}$
of ${\cal C}[\bbbx\times P(\bbbn)]$ is defined similarly from
acceptable enumerations
$({\cal W}^{\bbbx\times(\words)^m\times P(\bbbn)}_i)_{i\in\bbbn}$
of $\Sigma^0_1$ subsets of the $\bbbx\times(\words)^m$'s.
\medskip\\
{\bf 3.} Let $A\subseteq\bbbn$ and ${\cal C}^A$ be the $A$-oracle
syntactical class associated to ${\cal C}$.
\\
An enumeration $(W^{{\cal C}^A[\bbbx]}_i)_{i\in\bbbn}$ of
${\cal C}^A[\bbbx]$ is uniform if it is obtained from an
acceptable enumeration
$({\cal W}^{{\cal C}[\bbbx\times P(\bbbn)]}_i)_{i\in\bbbn}$ of
${\cal C}[\bbbx\times P(\bbbn)]$ by fixing the second order
argument. I.e.
$W^{{\cal C}^A[\bbbx]}_i=\{\ttx\in\bbbx : (\ttx,A)\in
{\cal W}^{{\cal C}[\bbbx\times P(\bbbn)]}_i\}$.
\end{definition}
%
\subsection{The $\min$ and $\max$ operators}
\label{ss:mimaxoperators}
%
The following definitions and results collect material
from \cite{ferbusgrigoKmaxKmin,ferbusgrigoMaxMinRecursion}.
\begin{definition}\label{def:PRminPRmax}
Let $\bbbx$ be some basic set.
We denote $\min$ and $\max$ the operators which map
partial functions $\varphi:\bbbx\times\bbbn\to\bbbn$
and partial functionals
$\Phi:\bbbx\times P(\bbbn)\times\bbbn\to\bbbn$
onto partial functions $\min\varphi, \max\varphi:D\to\bbbn$
and functionals $\min\Phi, \max\Phi:D\to\bbbn$ such that
\begin{eqnarray*}
(\min\varphi)(\ttx)&=&\min\{\varphi(\ttx,t):t\in\bbbn\
                 \wedge\ \varphi(\ttx,t)\mbox{ is defined}\}
\\
(\max\varphi)(\ttx)&=&\max\{\varphi(\ttx,t):t\in\bbbn\
                 \wedge\ \varphi(\ttx,t)\mbox{ is defined}\}
\\
(\min\Phi)(\ttx,A)&=&\min\{\Phi(\ttx,A,t):t\in\bbbn\
                 \wedge\ \phi(\ttx,t)\mbox{ is defined}\}
\\
(\max\Phi)(\ttx,A)&=&\max\{\Phi(\ttx,A,t):t\in\bbbn\
                 \wedge\ \phi(\ttx,t)\mbox{ is defined}\}
\end{eqnarray*}
with the convention that $\min\emptyset$ and $\max\emptyset$
and the $\max$ of an infinite set are undefined.
\medskip\\
{\bf 2.}
We let
\begin{eqnarray*}
\minpr[\bbbx\to\bbbn]&=&
\{\min\varphi : \varphi\in PR^{\bbbx\times\bbbn\to\bbbn}\}
\\
\minprA[\bbbx\to\bbbn]&=&
\{\min\varphi : \varphi\in PR^{\bbbx\times\bbbn\to\bbbn,A}\}
\\
\minpc[\bbbx\times P(\bbbn)\to\bbbn]&=&
\{\min\Phi : \Phi\in PC^{P(\bbbn)\times\bbbx\times\bbbn\to\bbbn}\}
\end{eqnarray*}
The classes $\maxpr[\bbbx\to\bbbn]$, $\maxprA[\bbbx\to\bbbn]$ and
$\maxpc[\bbbx\times P(\bbbn)\to\bbbn]$ are defined similarly
from the $\max$ operator.
\end{definition}
\begin{note}$\\ $
{\bf 1.}
Simple examples of functions in $\minpr[]$ are Kolmogorov
complexities $K$ and $H$. Examples of functions in $\maxpr[]$
are the Busy Beaver function and the (partial) function giving
the cardinal of $W_n$ (if finite).
\medskip\\
{\bf 2.}
The functional ${\cal K}(\ ||\ )$, defined in \S\ref{s:uniform},
is in $\minpc[P(\bbbn)\times\bbbn\to\bbbn]$.
\end{note}
Let's mention an easy result as concerns the syntactical complexity
of these functions.
\begin{proposition}\label{p:syntax}
Any function in $\minpr[\bbbx\to\bbbn]\cup\maxpr[\bbbx\to\bbbn]$
has $\Sigma^0_1\wedge\Pi^0_1$ graph.
The result extends to functionals and also relativizes.
\end{proposition}
\begin{proof}
Observe that $y=(\min\varphi)(\ttx)$ can be written
$$(\exists t\ y=\varphi(\ttx,t))\ \wedge\
(\forall t\ \forall s\ (\varphi(\ttx,t)\mbox{ converges in $s$
steps }\Rightarrow\ y\leq\varphi(\ttx,t)))$$
Idem for $y=(\max\varphi)(\ttx)$ with $\geq$ in place of $\leq$.
\end{proof}
We shall use the following straightforward corollary of the above
Proposition.
\begin{proposition}\label{p:MinPRsubsetPR0'}
All functions in
$\minpr[]$ and $\maxpr[]$ are partial recursive in $\emptyset'$.
\end{proposition}
An enumeration theorem holds for the families introduced
in Def.\ref{def:PRminPRmax}.
\begin{proposition}\label{p:enumPRminPRmax}
There exists an acceptable enumeration $(\phi_i)_{i\in\bbbn}$ of
$\minpr[\XN]$ (where acceptable means that the analogs of
conditions i--iii of Def.\ref{def:acceptable} hold.
In particular, the function $(i,\ttx)\mapsto\phi_i(\ttx)$ is
itself in $\minpr[\NXN]$).
\\
Idem with the class $\maxpr[\XN]$ and the functional classes
$\minpc[\bbbx\times P(\bbbn)\to\bbbn]$
and $\minpc[\bbbx\times P(\bbbn)\to\bbbn]$.
\end{proposition}
The following simple result about $\minpr[]$ and $\maxpr[]$
will be useful.
\begin{proposition}\label{p:composition}$\\ $
{\bf 1.}
If $\phi\in\minpr[\bbbx\to\bbbn]$ and $f:\bbby\to\bbbx$
is in $\PR[\bbby\to\bbbx]$ then
$\phi\circ f\in\minpr[\bbby\to\bbbn]$.\\
Idem with $\maxpr[]$ in place of $\minpr[]$.
\medskip\\
{\bf 2.}
If $\psi\in\minpr[\bbbn\to\bbbn]$ is monotone increasing
and $\phi\in\minpr[\bbbx\to\bbbn]$ then
$\psi\circ\phi\in\minpr[\bbbx\to\bbbn]$.
\end{proposition}
\begin{proof}
{\em 1.} Let $\phi(\ttx)=\min_t\varphi(\ttx,t)$ where $\varphi$
is partial recursive.
Then $\phi(f(\tty))=\min_t\varphi(f(\tty),t)$ is in
$\minpr[\bbby\to\bbbn]$ since $\varphi(f(\tty),t)$ is in
$\PR[\bbby\to\bbbn]$.
\medskip\\
{\em 2.} Let $\phi(\ttx)=\min_t\varphi(\ttx,t)$ and
$\psi(x)=\min_u\theta(x,u)$
where $\varphi,\theta$ are  partial recursive.
Since $\psi$ is monotone increasing, letting
$(\pi_1,\pi_2):\bbbn\to\bbbn^2$ be the inverse of Cantor
bijection, we have
\begin{eqnarray*}
\psi(\phi(\ttx))&=&\psi({\min}_t\varphi(\ttx,t))\\
&=&{\min}_t(\psi(\varphi(\ttx,t)))\\
&=&{\min}_t({\min}_u\theta(\varphi(\ttx,t),u))\\
&=&{\min}_v\theta(\varphi(\ttx,\pi_1(v)),\pi_2(v))
\end{eqnarray*}
which is in $\minpr[\bbbx\to\bbbn]$ since
$\theta(\varphi(\ttx,\pi_1(v)),\pi_2(v))$ is partial recursive.
\end{proof}
%
\section{Coimmunity and density}
\label{s:constructive}
%
%
\subsection{Constructive coimmunity and constructive density}
\label{ss:immune}
%
A classical result about Kolmogorov complexity $K$,
due to Barzdins (cf. \cite{livitanyi} Thm.2.7.1 iii, p.167,
or Zvonkin \& Levin, \cite{zvonkinlevin} p.92.),
states that if $\varphi$ is total recursive and tends to $+\infty$
then
\\\centerline{$\{x : K(x)<\varphi(x)\}$}
is an r.e. set which meets every infinite r.e. set,
i.e. $\{x : K(x)<\varphi(x)\}\cap W_i$ is an infinite
r.e. set whenever $W_i$ is infinite.
\\ (The case $\varphi$ is monotone increasing is due to
Kolmogorov, cf. \cite{zvonkinlevin} p.90,
or \cite{livitanyi} Thm.2.3.1 iii, p.119--120).
\\ In particular, $K$ has no total recursive unbounded
lower bound.
\medskip\\
In \S\ref{s:barzdins} we extend in various ways this result to sets
which are no more r.e. sets and involve Kolmogorov complexities
$\kmin[]$ or $\kmax[]$.
We also consider effectiveness of such properties in a sense
related to the notion of constructive immunity, first
considered in Xiang Li, 1983 \cite{lixiang}
(cf. Odifreddi's book \cite{odifreddi} p.267).
\begin{definition}\label{def:immune}
Let $(W_i)_{i\in\bbbn}$ be an acceptable enumeration of
recursively enumerable subsets of some basic set $\bbbx$.
\medskip\\
{\bf 1i.} (Dekker, 1958).
A set $X\subseteq\bbbx$ is {\em immune} if
it is infinite and contains no infinite r.e. set.
\medskip\\
{\bf 1ii.} (Xiang Li, 1983 \cite{lixiang}).
A set $X\subseteq\bbbx$ is {\em constructively immune} if
it is infinite and there exists some partial recursive function
$\varphi:\bbbn\to\bbbx$ such that
$$\forall i\ (W_i\ \mbox{\em is infinite }\Rightarrow
\varphi(i)\ \mbox{\em is defined }\wedge\
\varphi(i)\in W_i\setminus X)$$
{\bf 2i.}
A set $Z\subseteq\bbbx$ is $\Sigma^0_1$-dense if it contains
an infinite r.e. subset of any infinite r.e. set included
in $\bbbx$.
\medskip\\
{\bf 2ii.}
A set $Z\subseteq\bbbx$ is constructively $\Sigma^0_1$-dense if
there exists some total recursive function $\lambda$ such that
$$\forall i\ (W_i\ \mbox{\em is infinite }
\Rightarrow\
W_{\lambda(i)}\ \mbox{\em is an infinite subset of }
Z\cap W_i)$$
\end{definition}
\begin{note}\label{note:SigmaDense}
Rogers'Thm.\ref{thm:rogers} insures that the above notion of
constructive immunity and $\Sigma^0_1$-density do not depend
on the chosen enumeration of r.e. sets.
\end{note}
\begin{proposition}\label{p:immune}
$Z\subseteq\bbbx$ is constructively immune
if and only if it is infinite and its complement is
constructively $\Sigma^0_1$-dense.
\end{proposition}
\begin{proof}
$\Leftarrow$. Let $\varphi(i)$ be the point which appears
first in the enumeration of $W_{\lambda(i)}$ (of course,
$\varphi(i)$ is undefined in case $W_{\lambda(i)}$ is empty).
\medskip\\
$\Rightarrow$.
Define a partial recursive function $\mu(i,n)$ which
satisfies:
\\\indent - $\mu(i,0)=\varphi(i)$
\\\indent - $\mu(i,n+1)=\varphi(i_n)$
where $i_n$ is such that
      $W_{i_n}=W_i\setminus\{\mu(i,m) : m\leq n\}$
\\ Using the parametrization theorem, let $\lambda$ be total
recursive so that $W_{\lambda(i)}=\{\mu(i,m) : m\in\bbbn\}$.
If $W_i$ is infinite then all $\mu(i,m)$'s are defined
and distinct and belong to $W_i\cap Z$.
Thus, $W_{\lambda(i)}$ is an infinite subset of
$W_i\cap Z$.
\end{proof}
In case $Z$ is r.e., constructive $\Sigma^0_1$-density amounts
to say that $Z\cap W_i$ is infinite whenever $W_i$ is
infinite.
\medskip\\
Barzdin's result gives an instance of a constructively
$\Sigma^0_1$-dense r.e. set.
Other examples are maximal r.e. sets.
\begin{proposition}\label{p:maximal}
Any maximal r.e. set $Z$ is constructively $\Sigma^0_1$-dense.
\end{proposition}
\begin{proof}
Let $Z\subseteq\bbbx$ be r.e. where $\bbbx$ is some basic set.
We prove that for every infinite r.e. set $W_i\subseteq\bbbx$
the intersection $Z\cap W_i$ is also infinite.
\\
In fact, suppose $Z\cap W_i$ is finite. Then $W_i\setminus Z$ is
an infinite r.e. set disjoint from $Z$. Thus, $Z'=Z\cup W_i$ is
an r.e. set containing $Z$ such that the difference
$Z'\setminus Z=W_i\setminus Z$ is infinite.
Since $Z$ is maximal this implies that $Z'$ is cofinite.
Thus,
$$\bbbx\setminus Z=(\bbbx\setminus Z')\cup(W_i\setminus(Z\cap W_i))
=A\cup (W_i\setminus B)$$
where $A,B$ are finite sets. Hence $\bbbx\setminus Z$ is r.e.
and, consequently $Z$ is recursive. A contradiction.
\end{proof}
%
\subsection{Uniform constructive density}
\label{ss:uniformDensity}

In order to deal with Kolmogorov complexities
$K^{\emptyset'}, K^{\emptyset''},\ldots$ and their $Min/Max$
versions, we shall consider constructive density for
$\Sigma^0_n$ sets. This will be done through relativization
of $\Sigma^0_1$-density with respect to jump oracle
$\emptyset^{(n-1)}$.
\medskip\\
There is two natural ways to relativize $\Sigma^0_1$-density
to an oracle $A$ :
\begin{enumerate}
\item[$(*)$] Consider the $W^A_i$'s and ask for $\lambda$
       $A$-recursive.
\item[$(**)$] Consider the $W^A_i$'s and ask for $\lambda$
       recursive.
\end{enumerate}
The second way, which is the stronger one, will be the one
pertinent for applications to Kolmogorov complexities.
Of course, to deal with $(**)$, we must consider uniform
enumerations of $A$-r.e. sets and partial $A$-recursive
functions (cf. Prop.\ref{p:uniform}), i.e. we have to consider
the notion of constructive density with functionals.
This will, in fact, give {\em a strong version of $(**)$ in which
$\lambda$ is a total recursive function which does not depend
on $A$.}
\begin{definition}\label{def:SigmaDense}$\\ $
{\bf 1.}
Let ${\cal Z}\subseteq\bbbx\times P(\bbbn)$.
For $A\subseteq\bbbn$, let's denote
${\cal Z}^A=\{\ttx\in\bbbx : (\ttx,A)\in {\cal Z}\}$.
\\
Consider an acceptable enumeration $({\cal W}_i)_{i\in\bbbn}$
of $\Sigma^0_1$ subsets of $\bbbx\times P(\bbbn)$
(cf. Def.\ref{def:acceptable}) and let
$W^A_i=\{\ttx : (\ttx,A)\in{\cal W}_i\}$.
\\
${\cal Z}$ is constructively $\Sigma^0_1$-dense
if there exists some total {\em recursive} function $\lambda$
such that, for all $i\in\bbbn$ and all $A\in P(\bbbn)$,
\medskip\\
$(*)$\ \ $\forall i\in\bbbn\ \forall A\in P(\bbbn)\
(W^A_i \mbox{ \em is infinite}$

\hfill{$\Rightarrow\ W^A_{\lambda(i)}
           \mbox{ \em is an infinite subset of }
                    W^A_i\cap {\cal Z}^A)$}
\medskip\\
{\bf 2.}
$Z\subseteq\bbbx$ is constructively uniformly
$\Sigma^{0,A}_1$-dense if there exists some constructively
$\Sigma^0_1$-dense set ${\cal Z}\subseteq\bbbx\times P(\bbbn)$
such that $Z={\cal Z}^A$.
\\ In particular, there exists some total {\em recursive}
function $\lambda$ such that
\medskip\\
$(**)$\ \ $\forall i\in\bbbn\ (W^A_i \mbox{\em is infinite }
\Rightarrow\ W^A_{\lambda(i)}
           \mbox{\em is an infinite subset of }W^A_i$
\\
When $A=\emptyset^{(n-1)}$ we shall also say that $Z$ is
constructively uniformly $\Sigma^0_n$-dense.
\end{definition}
\begin{note}\label{note:SigmaDenseUniform}
Thm.\ref{thm:rogers} insures that the above notion of
constructive uniform $\Sigma^{0,A}_1$-density does not depend
on the chosen enumeration of $A$-r.e. sets, as long as it is
uniform, cf. Prop.\ref{p:uniform}.
\end{note}
\begin{remark}\label{rk:SigmaDense}
Using Point 2 of Prop.\ref{p:classicRec},
one can suppose that if $W^A_i$ is infinite then
$W^A_{\lambda(i)}$ is $A$-recursive and an $A$-r.e.
code for its complement is given by another total recursive
function $\lambda'$.
\end{remark}
\begin{note}\label{note:Sigma01dense}
In the vein of what we mentioned at the start of
\S\ref{ss:immune},
if $\varphi:\bbbn\to\bbbn$ is total $A$-recursive and tends
to $+\infty$ then Lemma.\ref{l:ggK} insures that
$\{x : K^A(x)<\varphi(x)\}$ is an $A$-r.e. set which
is uniformly constructively $\Sigma^{0,A}_1$-dense.
In case $\varphi(x)\infct\log(x)$, this set is coinfinite since
it excludes integers with incompressible binary representations.
\end{note}
\begin{remark}\label{rk:AImmune}
Immunity can also be relativized according to the different
policies $(*)$ and $(**)$. Also, Prop.\ref{p:immune} admits
straightforward extensions to the functional setting and the
uniform relativized one.
\end{remark}
Finally, let's observe that Prop.\ref{p:maximal} relativizes
in the uniform sense.
\begin{proposition}\label{p:Sigmamaximal}
Any maximal $A$-r.e. set $Z$ is uniformly constructively
$\Sigma^{0,A}_1$-dense.
\end{proposition}
\begin{proof}
Let $Z=\{\ttx : (\ttx,A)\in{\cal Z}\}$ where
${\cal Z}\subseteq\bbbx\times P(\bbbn)$ is $\Sigma^0_1$.
There is a total recursive function $\theta$ such that,
for all $A$ and $i$,
${\cal Z}\cap {\cal W}_i={\cal W}_{\theta(i)}$.
In particular, $Z\cap W^A_i=W^A_{\theta(i)}$ and
the argument of Prop.\ref{p:maximal} goes through.
\end{proof}
%
\subsection{Constructive $({\cal C},{\cal D})$-density}
\label{ss:DEDense}

Comparison of $K$ and $\kmin[],\kmax[]$ in the hierarchy theorem
\ref{thm:hierarchy} leads to a particular version of
constructive density applied to $\Sigma^0_1$ and to $\Pi^0_1$ sets
and involving subclasses
$\exists^{\leq\mu}(\Sigma^0_1\wedge\Pi^0_1)$ of $\Delta^0_2$ sets
described in Def.\ref{def:existsleq} below.
\medskip\\
We now introduce some central notions of this paper.
\begin{definition}\label{def:DEDense}
Let $\bbbx$ be a basic set.\\
{\bf 1i.}
Let ${\cal S},{\cal T}$ be families of subsets of $\bbbx$.
A set $Z\subseteq\bbbx$ is $({\cal S},{\cal T})$-dense if
for every infinite set $X\in{\cal S}$ the intersection
$Z\cap X$ contains an infinite subset $Y$ which is
in ${\cal T}$.
\medskip\\
{\bf ii.}
Let ${\cal S},{\cal T}$ be families of subsets of
$\bbbx\times P(\bbbn)$.
A set ${\cal Z}\subseteq\bbbx\times P(\bbbn)$ is
$({\cal S},{\cal T})$-dense if for every
${\cal X}\in{\cal S}$
there exists ${\cal Y}\in{\cal T}$
such that, for every $A$,
letting ${\cal X}^A=\{\ttx : (\ttx,A)\in{\cal X}\}$,
$${\cal X}^A\mbox{ is infinite }\Rightarrow\
{\cal Y}^A\mbox{ is infinite and included in }{\cal X}^A\cap{\cal Z}^A$$
{\bf 2.}
Let ${\cal C},{\cal D}$ be syntactical classes as in
Def.\ref{def:Dclass}.
\medskip\\
{\bf i.}
$Z$ is constructively $({\cal C},{\cal D})$-dense
if it is $({\cal C}[\bbbx],{\cal D}[\bbbx])$-dense in the sense of
1i above and, moreover, a ${\cal D}$-code for $Y$ can be
recursively obtained from a ${\cal C}$-code for $X$.
In other words, there exists some total recursive function
$\lambda:\bbbn\to\bbbn$ such that, for all $i$
$$W^{{\cal C}[\bbbx]}_i\
\mbox{ is infinite }
\ \Rightarrow\ \ W^{{\cal D}[\bbbx]}_{\lambda(i)}
\ \mbox{ is infinite and included in }
{W^{{\cal C}[\bbbx]}_i}\cap Z$$
{\bf ii.}
A set ${\cal Z}\subseteq\bbbx\times P(\bbbn)$ is
constructively $({\cal C},{\cal D})$-dense if it is
$({\cal C}[\bbbx],{\cal D}[\bbbx])$-dense in the sense of
1ii above and, moreover, an ${\cal D}$-code
for ${\cal Y}$ can be recursively obtained from a ${\cal C}$-code
for ${\cal X}$.
In other words, there exists some total recursive function
$\lambda:\bbbn\to\bbbn$ such that, for all $i$
\medskip\medskip\\
\indent$({\cal W}^A_i)^{{\cal C}[\bbbx\times P(\bbbn)]}\
\mbox{ is infinite }$\medskip

\hfill{$\Rightarrow\
({\cal W}^A_{\lambda(i)})^{{\cal D}[\bbbx\times P(\bbbn)]}
\ \mbox{ is infinite and included in }
({\cal W}^A_i)^{{\cal C}[\bbbx\times P(\bbbn)]}\cap {\cal Z}^A)$}
\end{definition}
\begin{note}$\\ $
{\bf 1.} Clearly, (constructive)
$(\Sigma^0_1,\Sigma^0_1)$-density
is exactly (constructive) $\Sigma^0_1$-density in the sense
of Def.\ref{def:SigmaDense}.
\medskip\\
{\bf 2.}
See Lemmas \ref{l:ggKmax}, \ref{l:ggKmin} for examples of
constructive
$(\Sigma^0_1,\exists^{\leq\mu}(\Sigma^0_1\wedge\Pi^0_1))$-density
and
$(\Pi^0_1,\exists^{\leq\mu}(\Sigma^0_1\wedge\Pi^0_1))$-density.
\end{note}
Let's state a simple result about $({\cal C},{\cal D})$-density.
\begin{proposition}\label{p:intersection}$\\ $
{\bf 1.}
The family of (constructively) $({\cal C},{\cal D})$-dense
subsets of $\bbbx$ (resp. $\bbbx\times P(\bbbn)$)
is superset closed.
\medskip\\
{\bf 2.}
Let $Z_1,Z_2\subseteq\bbbx$.
If $Z_1$ is (constructively) $({\cal C},{\cal D})$-dense
and $Z_2$ is (constructively) $({\cal D},{\cal E})$-dense
then $Z_1\cap Z_2$ is (constructively)
$({\cal C},{\cal E})$-dense.
\\
Idem for ${\cal Z}_1,{\cal Z}_2\subseteq\bbbx\times P(\bbbn)$.

\end{proposition}
\begin{proof}
Point 1 is obvious. As for point 2, let $X$ be an infinite set
in ${\cal C}[\bbbx]$.
Using $({\cal C},{\cal D})$-density of $Z_1$ we (recursively)
get (a code for) an infinite $X_1\subseteq X\cap Z_1$
in ${\cal D}$.
Then, using $({\cal D},{\cal E})$-density of $Z_2$,
we (recursively) get (a code for) an infinite
$X_2\subseteq X_1\cap Z_2\subseteq X\cap (Z_1\cap Z_2)$
in ${\cal E}$.
\\
For ${\cal Z}_1,{\cal Z}_2,{\cal X}\subseteq\bbbx\times P(\bbbn)$,
fix the second order argument $A$ and argue similarly with
${\cal Z}_1^A,{\cal Z}_2^A,{\cal X}^A$.
\end{proof}
%
%
\section{The $\often[]\ \! $ relations
      and the $\lless[]\ \!$ orderings}
\label{s:ll}
%
%
In this \S\ we introduce the central notions of this paper to
compare the growth of total functions $f,g:\bbbn\to\bbbn$.
%
%
\subsection{Relations {$\often[{\cal C},{\cal D}] {\cal F}$},
{$\often[{\cal C},{\cal D}] {{\cal F}\uparrow}$}
on maps $\bbbn\to\bbbn$}

\begin{definition}\label{def:oftless}
Let ${\cal C},{\cal D}$ be syntactical classes
(cf. Def.\ref{def:Dclass}) and ${\cal F}$ be a countable family of
functions $\bbbn\to\bbbn$ and $(\phi_i)_{i\in\bbbn}$ be a
(non necessarily injective) enumeration of ${\cal F}$
(in \S\ref{s:barzdins}, ${\cal F}$ will be $PR$ or $\minpr[]$,
cf. Def.\ref{def:PRminPRmax}).
\medskip\\
We let $\ f\often[{\cal C},{\cal D}] {\cal F} g\ $
(resp. $\ f\often[{\cal C},{\cal D}] {{\cal F}\uparrow} g\ $)
be the relation
between total functions $f,g:\bbbn\to\bbbn$
defined by the following conditions:
\begin{enumerate}
\item[i.]
For every total (resp. and monotone increasing) function
$\phi:\bbbn\to\bbbn$ in ${\cal F}$ which tends to $+\infty$,
the set $\ \{x : f(x)<\phi(g(x))\}\ $\ is
constructively $({\cal C},{\cal D})$-dense.
\item[ii.]
The constructive $({\cal C},{\cal D})$-density in
condition i is uniform in $\phi$ :
There exists some total recursive
$\lambda:\bbbn^2\to\bbbn$ such that, for all $i,j$,
\medskip\\
$\phi_i$ is total (resp. and monotone increasing)
and tends to $+\infty$
\\
$\wedge\ W^{\cal C}_j$ is infinite

\hfill{$\Rightarrow\ W^{\cal D}_{\lambda(i,j)}$
is an infinite subset of
$W^{\cal C}_j\cap\{x : f(x)<\phi_i(g(x))\}$}
\end{enumerate}
\end{definition}
\begin{remark}$\\ $
{\bf 1.}
The notation $\often[]\ $ stresses the fact that $f$ is
often much smaller than $g$ :
consider functions $\phi$ which are much smaller than
the identity function,
e.g. $\max(0,z-c)$, $\lfloor z/c\rfloor$,
$\lfloor\log(z)\rfloor$, $\log^*(z)$,\ldots
\medskip\\
{\bf 2.}
$\often[{\cal C},{\cal D}] {\cal F}$ carries the contents,
reformulated in terms of uniform constructive
$({\cal C},{\cal D})$-density, of Barzdins result cited above,
and that of adequate variants that we shall prove about
$\kmax[]$ and $\kmin[]$ (cf. Lemmas \ref{l:ggK},
\ref{l:ggKmax}, \ref{l:ggKmin}).
\medskip\\
{\bf 3.}
Suppose ${\cal F}$ contains all translation functions
$z\mapsto\max(0,z-c)$.
If $f\often[{\cal C},{\cal D}]{{\cal F}\uparrow}\ g$
(a fortiori if $f\often[{\cal C},{\cal D}]{\cal F}\ g$)
then $g$ is necessarily unbounded.
Else, if $c$ is a bound for $g$, consider $\phi(z)=\max(0,z-c)$
to get a contradiction.
\medskip\\
{\bf 4.}
$\often[{\cal C},{\cal D}]{{\cal F}\uparrow}$ is an extension of
$\often[{\cal C},{\cal D}]{\cal F}$ which has much better properties
(cf. Thm.\ref{thm:transitivity}).
\end{remark}
%
%
\subsection{Monotonicity versus recursive lower bound}

In case ${\cal F}=PR$, the monotonicity condition can be put
in another equivalent form.
\begin{proposition}
Relation $f \often[{\cal C},{\cal D}] {PR\uparrow} g$ holds
if and only if conditions i, ii in Def.\ref{def:oftless} hold for
every total functions $\phi,\phi_i:\bbbn\to\bbbn$ which recursively
tend to $+\infty$, i.e. there are recursive growth modulus
$\xi,\xi_i:\bbbn\to\bbbn$ such that
$$\forall N\ \forall n\geq \xi(N)\ \phi(n)\geq N
\ \ ,\ \
\forall N\ \forall n\geq \xi_i(N)\ \phi_i(n)\geq N$$
\end{proposition}
\begin{proof}
$\Rightarrow.$ If $\phi$ is total recursive and monotone increasing
and tends to $+\infty$ then it tends recursively to $+\infty$ : a
possible recursive growth modulus is
$$\xi(N)=\mbox{least $x$ such that $\phi(x)\geq N$}$$
$\Leftarrow.$ Observe that any total $\phi\in PR$ which tends
recursively to $+\infty$ has a total recursive minorant $\psi$
which also tends to $+\infty$, namely
$$\psi(0)=\varphi(0)\ \ ,\ \ \psi(N+1)=\varphi(\xi(1+\psi(N)))$$
where $\xi$ is a recursive growth modulus of $\varphi$.
\\
Of course, if true for $\psi$, conditions i, ii are also true for
$\phi$.
\end{proof}
%
\subsection{Transitivity} \label{sub:transitivity}
%
It is clear that if ${\cal C}\neq{\cal D}$ then
$\often[{\cal C},{\cal D}] {\cal F}$ and
$\often[{\cal C},{\cal D}]{{\cal F}\uparrow}$
may not be transitive, hence may not be orderings.
However, we have the following result.
\begin{theorem}[Transitivity theorem]
\label{thm:transitivity}$\\ $
{\bf 1.}
Let ${\cal B},{\cal C},{\cal D}$ be syntactical classes and
${\cal F},{\cal G}$ be countable classes of functions containing
the identity function $Id:\bbbn\to\bbbn$.
Then,
\medskip\medskip\\
$\begin{array}{cccccccccc}
i.&e&\often[{\cal B},{\cal C}] {\cal F}&f
&\often[{\cal C},{\cal D}] {\cal G}&g
&&&\Longrightarrow &
\begin{array}{rcl}
e & \often[{\cal B},{\cal D}] {\cal G} & g
\end{array}
\medskip\\
ii.&e&\often[{\cal B},{\cal C}]{{\cal F}\uparrow}&f
&\often[{\cal C},{\cal D}]{{\cal G}\uparrow}&g
&&&\Longrightarrow &\left\{
\begin{array}{rcl}
e & \often[{\cal B},{\cal D}]{{\cal G}\uparrow} & g
\medskip\\
e&\often[{\cal B},{\cal D}]{{\cal F}\uparrow}&g
\end{array}\right.
\end{array}$
\medskip\\
In case ${\cal F}$ is recursively closed by negative translation
of the output, i.e.
\\\centerline{$\phi\in{\cal F}\ \Rightarrow\
\forall c\ \max(0,\phi-c)\in{\cal F}$}
and there exists a total recursive function $\theta:\bbbn^2\to\bbbn$
such that
\\\centerline{$\max(0,\phi_i-c)=\phi_{\theta(i,c)}$}
then
\medskip\\
$\begin{array}{cccccccccccc}
iii.&e & \leqct & f & \often[{\cal C},{\cal D}] {\cal F}&g
&&& \Longrightarrow & e & \often[{\cal C},{\cal D}] {\cal F}& g
\end{array}$
\medskip\\
In case ${\cal F}$ is recursively closed by negative translation
of the output and also by negative translation of the input, i.e.
\\\centerline{$\phi\in{\cal F}\ \Rightarrow\
\forall c\ x\mapsto \phi(\max(0,x-c))\in{\cal F}$}
and there exists a total recursive function $\zeta:\bbbn^2\to\bbbn$
such that, for all $x$,
\\\centerline{$\phi_i(\max(0,x-c))=\phi_{\zeta(i,c)}(x)$}
then
\medskip\medskip\\
$\begin{array}{cccccccccccc}
iv.&e&\leqct&f&\often[{\cal C},{\cal D}]{{\cal F}\uparrow}
& g & \leqct & h
&\Longrightarrow&e&\often[{\cal C},{\cal D}]{{\cal F}\uparrow}&h
\end{array}$
\medskip\medskip\\
{\bf 2.}
{\em Case ${\cal C}={\cal D}$.}
Relations $\ \often[{\cal C},{\cal C}] {\cal F}$ and
$\often[{\cal C},{\cal C}]{{\cal F}\uparrow}$
are strict orderings.
\end{theorem}
\begin{proof}
{\em 1i.} Suppose
$e\often[{\cal B},{\cal C}] {\cal F} f
\often[{\cal C},{\cal D}] {\cal G}g$.
Observe that
$$\{x : e(x)<f(x)\}\ \cap\ \{x : f(x)<\phi(g(x))\}\
\subseteq\ \{x : e(x)<\phi(g(x))\}$$
Since $Id\in{\cal F}$, the sets on the left are respectively
constructively $({\cal B},{\cal C})$-dense and
$({\cal C},{\cal D})$-dense, uniformly in $\phi$.
Applying Prop.\ref{p:intersection}, we see that
$\{x : e(x)<\phi(g(x))\}$ is constructively
$({\cal B},{\cal D})$-dense,
whence $e\often[{\cal B},{\cal D}] {\cal G} g$.
\medskip\\
{\em 1ii.}
The above argument also gives
$e\often[{\cal B},{\cal D}]{{\cal G}\uparrow} g$.
To get $e\often[{\cal B},{\cal D}]{{\cal F}\uparrow} g$,
argue as above and observe that
$$\{x : e(x)<\phi(f(x))\}\ \cap\ \{x : f(x)<g(x)\}\
\subseteq\ \{x : e(x)<\phi(g(x))\}$$
whenever $\phi$ is monotone increasing.
\medskip\\
{\em 1iii.} Let $c$ be such that $e(x)\leq f(x)+c$
for all $x$.
\\
If $\phi\in{\cal F}$ is total and tends to $+\infty$, so is its
negative output translation
$$\widehat{\phi}_c(z)=\max\{0,\phi(z)-c\}$$
Suppose $f(x)<\widehat{\phi}_c(g(x))$.
Then $\widehat{\phi}_c(g(x))>0$ so that
\medskip\\\centerline{$\begin{array}{rclcl}
&&\widehat{\phi}_c(g(x))&=&\phi(g(x))-c\\
&&f(x)&<&\phi(g(x))-c\\
e(x)&\leq&f(x)+c&<&\phi(g(x))
\end{array}$}
This proves the following inclusion
\begin{eqnarray}\label{inclusion0}
\{x : f(x)<\widehat{\phi}_c(g(x))\}
                 \subseteq \{x : e(x)<\phi(g(x)\}
\end{eqnarray}
Relation $f\often[{\cal C},{\cal D}] {\cal F}g$ insures that
$\{x : f(x)<\widehat{\phi}_c(g(x))\}$ is constructively
$({\cal C},{\cal D})$-dense.
Inclusion (\ref{inclusion0}) implies that the same is true with
$\{x : e(x)<\phi(g(x))\}$.
Since a code for $\widehat{\phi}_c$ is recursively obtained from
a code for $\phi$, this proves
$e\often[{\cal C},{\cal D}] {\cal F}g$.
\medskip\\
{\em 1iv.} Let $c$ be now such that $e(x)\leq f(x)+c$ and
$g(x)\leq h(x)+c$ for all $x$.
\\
If $\phi\in{\cal F}$ is total, monotone increasing
and tends to $+\infty$, so is its negative input and output
translation
$$\widetilde{\phi}_c(z)=\max\{0,\phi(\max(0,z-c))-c\}$$
Suppose $f(x)<\widetilde{\phi}_c(g(x))$.
Then $\widetilde{\phi}_c(g(x))>0$ so that
\medskip\\\centerline{$\begin{array}{rclcl}
&&\widetilde{\phi}_c(g(x))&=&\phi(\max(0,g(x)-c)-c\\
&&f(x)&<&\phi(\max(0,g(x)-c)-c\\
e(x)&\leq&f(x)+c&<&\phi(\max(0,g(x)-c)
\end{array}$}
\medskip\\
Now, $g(x)\leq h(x)+c$ and $\phi$ is monotone increasing, hence
$$e(x)<\phi(\max(0,g(x)-c)\leq\phi(\max(0,h(x))=\phi(h(x))$$
This proves inclusion
\begin{eqnarray}\label{inclusion00}
\{x : f(x)<\widetilde{\phi}_c(g(x))\}
                 \subseteq \{x : e(x)<\phi(h(x)\}
\end{eqnarray}
Relation $f\often[{\cal C},{\cal D}] {{\cal F}\uparrow}g$ insures that
$\{x : f(x)<\widetilde{\phi}_c(g(x))\}$ is constructively
$({\cal D},{\cal E})$-dense.
Inclusion (\ref{inclusion00}) implies that the same is true with
$\{x : e(x)<\phi(h(x))\}$.
Since a code for $\widetilde{\phi}_c$ is recursively obtained from
a code for $\phi$, this proves
$e\often[{\cal C},{\cal D}]{{\cal F}\uparrow}h$.
\medskip\\
{\em 2.} Transitivity of $\often[{\cal C},{\cal C}] {\cal F}$ and
$\often[{\cal C},{\cal C}]{{\cal F}\uparrow}$
is an obvious consequence of 1i--ii.
As for irreflexivity, arguing with $\phi=Id$
(which is in ${\cal F}$), we see that $f(x)<\phi(f(x))$ is
impossible, so that $f\often[{\cal C},{\cal C}]{\cal F}f$ and
$f\often[{\cal D},{\cal C}]{{\cal F}\uparrow}f$ are always
false. Thus, $\often[{\cal C},{\cal C}] {\cal F}$ and
$\often[{\cal C},{\cal C}]{{\cal F}\uparrow}$
are {\em strict} orderings.
\end{proof}
%
\subsection{Orderings $\lless[{\cal C},{\cal D}] {\cal F}$,
$\lless[{\cal C},{\cal D}]{{\cal F}\uparrow}$
on maps $\bbbn\to\bbbn$}
\label{ss:ll}

Points iii-iv of the above theorem show that taking intersection
with the ``up to a constant" ordering $\leqct$ transforms the
relations
$\often[{\cal C},{\cal D}] {\cal F}$ and
$\often[{\cal C},{\cal D}]{{\cal F}\uparrow}$ into strict orderings
$\lless[{\cal C},{\cal D}] {\cal F}$ and
$\lless[{\cal C},{\cal D}]{{\cal F}\uparrow}$.
\begin{definition}\label{def:ll}
$\lless[{\cal C},{\cal D}] {\cal F}$ and
$\lless[{\cal C},{\cal D}]{{\cal F}\uparrow}$ are the intersections
of the $\often[{\cal C},{\cal D}] {\cal F}$ and
$\often[{\cal C},{\cal D}]{{\cal F}\uparrow}$ relations
with the $\leqct$ ordering on total maps $\bbbn\to\bbbn$.
\end{definition}

\begin{theorem}[{$\lless[{\cal C},{\cal D}] {\cal F}$} and
{$\lless[{\cal C},{\cal D}]{{\cal F}\uparrow}$} are strict
orderings]\label{thm:order}$\\ $
Let ${\cal A},{\cal B},{\cal C},{\cal D}$ be syntactical
classes and let ${\cal F},{\cal G}$ be countable classes of
functions $\bbbn\to\bbbn$ which contain $Id$ and which are
recursively closed by output and input translation
(cf. Thm.\ref{thm:transitivity}) relative to some enumerations of
${\cal F},{\cal G}$.
Then
\medskip\medskip\\
\indent$\begin{array}{lccccccccr}
i.&e &\lless[{\cal B},{\cal C}] {\cal F}&
f&\lless[{\cal C},{\cal D}] {\cal G}& g
&&& \Rightarrow &
\begin{array}{lcr}
e& \lless[{\cal B},{\cal D}] {\cal G}& g
\end{array}
\medskip\\
ii.&e &\lless[{\cal B},{\cal C}] {{\cal F}\uparrow}&
f&\lless[{\cal C},{\cal D}] {{\cal G}\uparrow}& g
&&& \Rightarrow &
\left\{
\begin{array}{lcr}
e& \lless[{\cal B},{\cal D}] {{\cal F}\uparrow}& g
\medskip\\
e& \lless[{\cal B},{\cal D}] {{\cal G}\uparrow}& g
\end{array}\right.
\end{array}$
\medskip\\
\indent$\begin{array}{lccccccccccc}
iii.&e& \leqct& f &\lless[{\cal C},{\cal D}] {\cal G}& g
&&& \Rightarrow & e &\lless[{\cal C},{\cal D}] {\cal G}& g
\medskip\\
iv.&e&\leqct&f&\lless[{\cal C},{\cal D}] {{\cal G}\uparrow}
& g& \leqct& h
&\Rightarrow&e&\lless[{\cal C},{\cal D}] {{\cal G}\uparrow}&h
\end{array}$
\medskip\medskip\\
In particular, properties iii and iv can be applied with $\leqct$
replaced by $\lless[{\cal A},{\cal B}] {\cal F}$ or
$\lless[{\cal A},{\cal B}] {{\cal F}\uparrow}$.
\medskip\\
Relations $\lless[{\cal C},{\cal D}] {\cal F}$ and
$\lless[{\cal C},{\cal D}] {{\cal F}\uparrow}$ are strict orderings
such that
\medskip\medskip\\
\indent$\begin{array}{lccccc}
v.&f{\lless[{\cal C},{\cal D}] {\cal F}}g &\Rightarrow &
f{\lless[{\cal C},{\cal D}] {{\cal F}\uparrow}}g &
\Rightarrow & f\infct g
\end{array}$\medskip
\end{theorem}
\begin{proof}
Conditions i--iv are straightforward consequences of the similar
conditions in Thm.\ref{thm:transitivity}.
\medskip\\
Condition i--ii yields transitivity of
$\lless[{\cal C},{\cal D}] {\cal F}$ and
$\lless[{\cal C},{\cal D}] {{\cal F}\uparrow}$.
\medskip\\
Implication
$f{\lless[{\cal C},{\cal D}] {\cal F}}g\ \Rightarrow\
f{\lless[{\cal C},{\cal D}] {{\cal F}\uparrow}}g$ is trivial.
Let's prove that $\lless[{\cal C},{\cal D}] {\cal F}$ refines
$\infct$ (and not merely $\leqct$).\\
Suppose $f\lless[{\cal C},{\cal D}] {\cal F}g$. Then $f\leqct g$.
Also, letting $\psi(z)=\max(0,z-c)$, we see that
$\{x : f(x)<\psi(g(x)\}=\{x : f(x)<g(x)-c\}$ is infinite,
hence the condition $\ \forall x\ g(x)\leq f(x)+c\ $
is impossible, whatever be $c$.
Thus, $f\infct g$.
\end{proof}
The above theorem shows that composition of the orderings
$\lless[{\cal C},{\cal D}] {\cal F}$ and
$\lless[{\cal C},{\cal D}] {{\cal F}\uparrow}$
is remarkably flexible. In particular,
\begin{corollary}\label{coro:order}
If $1\leq i\leq n$ and $1\leq j<k\leq m\leq n$ then
\medskip\medskip\\\centerline{
$\begin{array}{ccccccccccccc}
f_0&\lless[{\cal D}_1,{\cal E}_1] {{\cal F}_1\uparrow}
& \ldots&\lless[{\cal D}_n,{\cal E}_n] {{\cal F}_n\uparrow}& f_n
&\Rightarrow& f_0&\lless[{\cal D}_i,{\cal E}_i] {{\cal F}_i\uparrow}
& f_n
\medskip\\
f_0 &\lless[{\cal C}_1,{\cal C}_2] {{\cal F}_2\uparrow}
& \ldots&\lless[{\cal C}_{n-1},{\cal C}_n] {{\cal F}_n\uparrow}& f_n
&\Rightarrow&f_0&\lless[{\cal C}_j,{\cal C}_m] {{\cal F}_k\uparrow}& f_n
\end{array}$}
\end{corollary}
%
%
\subsection{Left composition and
$\lless[{\cal C},{\cal D}] {\cal F}$}
\label{ss:variant}

Def.\ref{def:oftless}, \ref{def:ll} compare total functions
$f,g:\bbbn\to\bbbn$ via the associated sets
$\{x:f(x)<\phi(g(x))\}$ for $\phi\in{\cal F}$.
One could also compare $f,g$ via the sets
$\{x:\phi(f(x))<g(x)\}$ for $\phi\in{\cal F}$.
Similar properties could be derived.
\medskip\\
Though we shall not use it in the sequel, there is a property
of $\lless[{\cal C},{\cal D}] {\cal F}$ and
$\lless[{\cal C},{\cal D}] {{\cal F}\uparrow}$ which is
interesting on its own and gives an alternative definition of
$\lless[{\cal C},{\cal D}] {\cal F}$ and
$\lless[{\cal C},{\cal D}] {{\cal F}\uparrow}$ where the
inequality $f(x)<\phi(g(x))$ gets a symmetric form
$\psi(f(x))<\phi(g(x))$ involving functions
$\psi,\phi$ on both sides of the inequality.
We prove it in case ${\cal F}$ is $PR$ or $\minpr[]$.
\begin{proposition}\label{p:variant}
Let ${\cal C},{\cal D}$ be syntactical classes and
${\cal F}$ be $PR$ or $\minpr[]$.
Let $\psi:\bbbn\to\bbbn$ be a total recursive function. Then
\medskip\medskip\\\medskip\centerline{
$\begin{array}{ccccccc}
f&\lless[{\cal C},{\cal D}] {\cal F}&g &\Rightarrow&
\psi\circ f&\lless[{\cal C},{\cal D}] {\cal F}&g
\medskip\medskip\\
f&\lless[{\cal C},{\cal D}] {{\cal F}\uparrow}& g &\Rightarrow&
\psi\circ f&\lless[{\cal C},{\cal D}] {{\cal F}\uparrow}&g
\end{array}$}
Moreover, the constructive density afferent to
the relations $\psi\circ f\lless[{\cal C},{\cal D}] {\cal F}g$
and $\psi\circ f\lless[{\cal C},{\cal D}] {{\cal F}\uparrow}g$
is uniform in $\psi$.
\end{proposition}
\begin{proof}
1. Let $\phi:\bbbn\to\bbbn$ be a total function in ${\cal F}$
which tend to $+\infty$.
We prove that $\{x : \psi(f(x))<\phi(g(x))\}$ is constructively
$({\cal C},{\cal D})$-dense.
\\
Set $\psi'(z)=\max(z,max\{\psi(u) : u\leq z\})$.
Then $\psi'\geq Id$ is total recursive, monotone increasing and
unbounded.
Since $\psi\leq\psi'$, we have
\begin{eqnarray}\label{cond:phiprime}
\{x:\psi'(f(x))<\phi(g(x))\}&\subseteq&\{x:\psi(f(x))<\phi(g(x))\}
\end{eqnarray}
2. Define $\alpha,\zeta:\bbbn\to\bbbn$ as follows:
\begin{eqnarray*}
\alpha(z) &=& \mbox{largest }u\mbox{ such that }
                           \psi'(u)\leq\phi(z)\\
\zeta(z) &=& \mbox{smallest }s\mbox{ such that }\psi'
\mbox{ is constant on }[s,\alpha(z)]
\end{eqnarray*}
Since $\phi(z)$ and $\psi'(z)$ tend to $+\infty$ so do
$\alpha(z)$ and $\zeta$. Also,
\begin{eqnarray}\label{cond:zeta}
\forall u<\zeta(z)\ \ \psi'(u)<\psi'(\zeta(z))
=\psi'(\alpha(z))\leq\phi(z)
\end{eqnarray}
Finally, $\alpha$ and $\zeta$ are in ${\cal F}$.
If ${\cal F}=PR$, this is trivial.
If ${\cal F}=\minpr[]$ and $\phi(x)=\min_t\phi_t(x)$ then
observe that $\alpha(x)=\min_t\alpha_t(x)$ and
$\zeta(x)=\min_t\zeta_t(x)$ (where $\alpha_t,\zeta_t$ are defined
from $\psi,\phi_t$ as are $\alpha,\zeta$ from $\psi,\phi$).
\medskip\\
3. Condition (\ref{cond:zeta}) applied to $z=g(x), u=f(x)$
insures
\medskip\\\medskip\centerline{$f(x)<\zeta(g(x))\
\Rightarrow\ \psi'(f(x))<\phi(g(x))$}
whence
\begin{eqnarray}\label{cond:inclusion}
\{x : f(x)<\zeta(g(x))\}
          & \subseteq & \{x : \psi'(f(x))<\phi(g(x))\}
\end{eqnarray}
Condition $f\often[{\cal C},{\cal D}]{\cal F} g$ applied
to $\zeta$
insures that $\{x : f(x)<\zeta(g(x))\}$ is constructively
$({\cal C},{\cal D})$-dense.
Using inclusions (\ref{cond:inclusion}) and
(\ref{cond:phiprime}),
we see that so is $\{x : \psi(f(x))<\phi(g(x))\}$.
\medskip\\
4. In case $f\often[{\cal C},{\cal D}]{{\cal F}\uparrow} g$,
then $\phi$ is monotone
increasing. Since $\psi$ is also monotone increasing, so are
$\alpha,\zeta$ and we get
$\psi\circ f\often[{\cal C},{\cal D}]{\cal F} g$.
\medskip\\
5. Finally, observe that all the construction is uniform in
$\psi$ and $\phi$.
\end{proof}
%
%
\section{Functional Kolmogorov complexity}
\label{s:uniform}
%
The purpose of this section is to reconsider the oracular
version of Kolmogorov complexity. We shall view the oracle
as a parameter in a second order variant of conditional
Kolmogorov complexity.
%
\subsection{Kolmogorov complexity of a functional}
\label{ss:Kuniform}
%
\begin{definition}\label{def:KsubFunctional}
Let $\bbbx$ be a basic set.
\\
The Kolmogorov complexity
${\cal K}_F:\bbbx\times P(\bbbn)\to\bbbn$
associated to a partial functional
$F:\words\times P(\bbbn)\to\bbbx$ is defined as follows:
\medskip\\\centerline{$\begin{array}{lcl}
{\cal K}_F(\ttx\,||\,A)
&=&\mbox{smallest $|\ttp|$ such that }(F(\ttp,A)=\ttx)
\end{array}$}
\end{definition}
\begin{note}$\\ $
{\bf 1.}
Forgetting the $A$, we get the
classical notion $K_F(\ttx)$ with $F:\words\to\bbbx$.
Freezing the $A$ also leads to the classical oracular notion.
This is the contents of the next obvious proposition
and of Thm.\ref{thm:unifVersusOracle} below.
\medskip\\
{\bf 2.}
The double bar $||$ is used so as to get no confusion with
usual conditional Kolmogorov complexity where the condition
is a first-order object.
\medskip\\
{\bf 3.}
The above definition can obviously be extended to conditional
Kolmogorov complexity ${\cal K}_F(\ttx\,|\,\tty\,||\,A)$
where $F:\words\times\bbby\times P(\bbbn)\to\bbbx$.
\end{note}
\begin{proposition}\label{p:uniformCond}
Let $F$ be as in Def.\ref{def:KsubFunctional}.
For $A\in P(\bbbn)$, denote
\\\centerline{$F^A:\words\to\bbbx$}
the function such that $F^A(\ttp)=F(\ttp,A)$.
Then, for all $\ttx\in\bbbx$,
$$K_{F^A}(\ttx)={\cal K}_F(\ttx\,||\,A)$$
\end{proposition}
%
\subsection{Functional invariance theorem}
\label{ss:uniformInvariance}
%
The usual proof of the invariance theorem
(Kolmogorov, 1965 \cite{kolmo65})
extends easily when considering partial computable functionals
$\words\times P(\bbbn)\to\bbbn$ in place of partial recursive
functions $\words\to\bbbn$, leading to what we call
{\em functional Kolmogorov complexity} and denote
${\cal K}(\ttx\,||\,A)$.
\begin{theorem}
[Functional Invariance Theorem]\label{thm:uniformInvariance}$\\ $
{\bf 1.}
Let ${\cal F}$ be the family of partial computable functionals
$\words\times P(\bbbn)\to\bbbx$.
When $F$ varies in ${\cal F}$, there is a least $K_F$
up to an additive constant:
$$\exists F\in {\cal F}\ \ \forall G\in {\cal F}\ \ \
{\cal K}_F\leqct {\cal K}_G$$
Such an $F$ is said to be optimal in ${\cal F}$.
We let ${\cal K}(\ ||\ )$ be ${\cal K}_F$ where $F$ is some fixed
optimal functional.
\medskip\\
{\bf 2.}
Let $(F_k)_{k\in\bbbn}$ be a partial computable enumeration of
$PC^{\words\times P(\bbbn)\to\bbbx}$.
Let ${\cal U}:\words\times P(\bbbn)\to\bbbx$ be such that
$${\cal U}(0^k1\ttp,A)=F_k(\ttp,A)\ \ \ \ \ \ \ \
{\cal U}(0^k)=0$$
Then ${\cal U}$ is optimal in $PC^{\words\times P(\bbbn)\to\bbbx}$.
\end{theorem}
\begin{proof}
It clearly suffices to prove Point 2.
The usual proof of the classical invariance theorem gives indeed
the functional version stated above.
\begin{eqnarray*}
{\cal K}_{F_k}(\ttx\,||\,A)&=&\min\{|\ttp|:F_k(\ttp,A)=\ttx\}\\
&=&\min\{|\ttp|:{\cal U}(0^k1\ttp,A)=\ttx\}\\
&=&\min\{|0^k1\ttp|-k-1:{\cal U}(0^k1\ttp,A)=\ttx\}\\
&\geq&\min\{|\ttq|-k-1:{\cal U}(\ttq,A)=\ttx\}\\
&=&\min\{|\ttq|:{\cal U}(\ttq,A)=\ttx\}-k-1\\
&=&{\cal K}_{\cal U}(\ttx\,||\,A)-k-1
\end{eqnarray*}
Whence $\ {\cal K}_{\cal U}\leq\ {\cal K}_{F_k}+k+1$ and therefore
$\ {\cal K}_{\cal U}\ \leqct\ {\cal K}_{F_k}$.
\end{proof}
\begin{remark}\label{rk:uniformCond}$\\ $
{\bf 1.}
Obviously, ${\cal K}_F(\ttx\,||\,A)$ does depend on $A$.
For example, if $\ttx\in\bbbn$ is incompressible then
${\cal K}_F(\ttx\,||\,\emptyset)\eqct\log(\ttx)$ whereas
${\cal K}_F(\ttx\,||\,\{\ttx\})\eqct0$.
\\
The contents of the functional invariance theorem is that,
for some $F$'s (the optimal ones) the number
\\\centerline{$\max\{{\cal K}_F(\ttx\,||\,A)
-{\cal K}_G(\ttx\,||\,A):\ttx\in\bbbn,A\in P(\bbbn)\}$}
is finite for any given $G$.
\medskip\\
{\bf 2.}
For the functional invariance theorem, we only have to suppose the
enumeration $(F_k){k\in\bbbn}$ to be partial computable as a
functional $\bbbn\times\words\times P(\bbbn)\to\bbbx$. There is
no need that it be acceptable (cf. Def.\ref{def:acceptable}).
\end{remark}
As for the usual Kolmogorov complexity, computable approximation
from above is possible.
\begin{proposition}\label{p:approxK}
There exists a total computable functional
$$(\ttx,t,A)\in\bbbx\times P(\bbbn)\times\bbbn
\mapsto{\cal K}^t(\ttx\,||\,A)$$
which is decreasing with respect to $t$ and such that,
for all $\ttx,A$,
$${\cal K}(\ttx\,||\,A)=\min\{{\cal K}^t(\ttx\,||\,A):t\in\bbbn\}$$
\end{proposition}
\begin{proof}
Letting ${\cal K}={\cal K}_{\cal U}$ where
${\cal K}\in PC^{\words\to\bbbn}$, set
\begin{eqnarray*}
B(\ttx,t,A)&=&\{|\ttp| : |\ttp|\leq t\ \wedge\ {\cal U}(\ttp,A)=\ttx\
\wedge\ {\cal U}(\ttp,A)\mbox{ halts in $\leq t$ steps}\}
\\
T(\ttx,A)&=&\mbox{smallest $t$ such that }B(\ttx,t,A)\neq\emptyset
\\
{\cal K}^t(\ttx\,||\,A)&=&\mbox{smallest }|\ttp|\in
B(\ttx,t,A)\cup B(\ttx,T(\ttx,A),A)
\end{eqnarray*}
\end{proof}
%
%
\section{The $Min/Max$ hierarchy of Kolmogorov complexities}
\label{s:KminKmax}
%
Infinite computations in relation with Kolmogorov complexity
were first considered in Chaitin, 1976 \cite{chaitin76higher} and
Solovay, 1977 \cite{solovay77}.
Becher \& Daicz \& Chaitin, 2001 \cite{becherchaitindaicz},
introduced a variant $\hi$ of the prefix version of Kolmogorov
complexity by allowing programs leading to possibly infinite
computations but finite output (i.e. remove the sole halting
condition).
This variant satisfies $H^{\emptyset'}\infct\hi\infct H$
(cf. \cite{becherchaitindaicz,bechernies}).
\\
In \cite{ferbusgrigoKmaxKmin}, 2004, we introduced a machine-free
definition $\kmax[]$ of the usual (non prefix) Kolmogorov version
$\ki$, together with a dual version $\kmin[]$.
The proof in \cite{bechernies} of the above
inequalities extends easily to the $K$ setting for $\kmax[]$.
However, a different argument is required in order to get
the $\kmin[]$ version (cf. \cite{ferbusgrigoKmaxKmin}).
%
\subsection{$Min/Max$ Kolmogorov complexities}
\label{ss:KminKmax}
%
The following definitions and theorems collects material
from \cite{ferbusgrigoKmaxKmin}.
The classical way to define Kolmogorov complexity extends
directly to these classes.
\begin{theorem}[$Min/Max$ Invariance theorem]
\label{thm:invarianceMinMax}
$\\ $
{\bf 1.}
Let ${\cal F}$ be $\minpr[\WN]$ or $\maxpr[\WN]$
(cf. Def.\ref{def:PRminPRmax}).
When $\phi$ varies in ${\cal F}$ there is a least $K_\phi$,
up to an additive constant (cf. Notation \ref{not:start}):
\begin{eqnarray*}
&\exists\phi\in \minpr[\WN]\ \
\forall\psi\in \minpr[\WN]\
\ \ K_\phi\leqct K_\psi&
\medskip\\
&\exists\phi\in \maxpr[\WN]\ \
\forall\psi\in \maxprA[\WN]\
\ \ K_\phi\leqct K_\psi&
\end{eqnarray*}
Such $\phi$'s are said to optimal in ${\cal F}$.
\\
We let
\\- $\kmin[]$ denote $K_\phi$ where $\phi$ is any function
optimal in $\minpr[\WN]$,
\\- $\kmax[]$ denote $K_\phi$ where $\phi$ is any function
optimal in $\maxpr[\WN]$.
\medskip\\
{\bf 2.}
Suppose $(\phi_k)_{k\in\bbbn}$ is an enumeration of
$\minpr[\WN]$ such that the function
$(k,\ttp)\mapsto\phi_k(\ttp)$ is in $\minpr[\NWN]$.
Let $U_{\min}$ be such that
$$U_{\min}(0^k1\ttp)=\phi_k(\ttp)\ \ \ \ \ \ \ \
U_{\min}(0^k)=\phi_k(\lambda)$$
Then $U_{\min}$ is optimal in $\minpr[\WN]$.
\\
Idem with $\maxpr[\WN]$.
\medskip\\
{\bf 3.}
Relativizing to an oracle $A\subseteq\bbbn$, one similarly
defines $\kmin[A]$ and $\kmax[A]$ and the analog of Point 2
also holds.
\end{theorem}
\begin{remark}[\cite{ferbusgrigoKmaxKmin}]
There exists optimal functions for $\maxpr[\WN]$
of the form $\max f$ where $f:\words\times\bbbn\to\bbbn$
is {\em total recursive}.
\\
This is false for $\minpr[\WN]$.
\end{remark}
Relativizing to the successive jumps oracles, we get an
infinite family of Kolmogorov complexities for which holds
a hierarchy theorem.
\begin{theorem}[The $Min/Max$ Kolmogorov hierarchy,
\cite{ferbusgrigoKmaxKmin}]\label{thm:hierarchyInfct}
$$\log \supct K \supct
\begin{array}{c}
\kmin[]\\ \kmax[]
\end{array}
\supct K^{\emptyset'}\supct
\begin{array}{c}
K^{\emptyset'}_\submin
\medskip\\ K^{\emptyset'}_\submax
\end{array}
\supct K^{\emptyset''}\supct
\begin{array}{c}
K^{\emptyset''}_\submin
\medskip\\ K^{\emptyset''}_\submax
\end{array}
\supct K^{\emptyset'''}...$$
\end{theorem}
Strict inequalities
$K\supct\kmax[]\supct K^{\emptyset'}
\supct\kmax[\emptyset']\supct K^{\emptyset''}$
were first proved by Becher \& Chaitin, 2001--2002
\cite{becherchaitindaicz} (for the prefix variants).
\medskip\\
The main application of the
$\lless[{\cal C},{\cal D}]{\cal F}$ and
$\lless[{\cal C},{\cal D}]{{\cal F}\uparrow}$ orderings
introduced in \S\ref{s:ll} is a strong improvement of this
hierarchy theorem, cf. Thm.\ref{thm:hierarchy}.
\medskip\\
Finally, we shall need the following result
(cf. \cite{ferbusgrigoKmaxKmin}, or \cite{becherchaitindaicz}
as concerns $\kmax[]$).
\begin{theorem}\label{thm:KminRecIn0'}
$K,\kmin[],\kmax[]$ are recursive in $\emptyset'$.
\end{theorem}
%
\subsection{Functional $Min/Max$ Kolmogorov complexities}
\label{ss:uniformKmiKmax}
%
The Invariance Theorems for $\maxpr[]$ and $\minpr[]$
(cf. Thm.\ref{thm:invarianceMinMax}) admit functional versions,
the proofs of which are exactly the same as that in
Thm.\ref{thm:uniformInvariance}.
\begin{theorem}[$Min/Max$ Functional Invariance Theorem]
\label{thm:MaxMinUniformInvariance}$\\ $
{\bf 1.}
When $F:\words\times P(\bbbn)\to\bbbn$ varies over
$\minpc[\words\times P(\bbbn)\to\bbbn]$ or over
$\maxpc[\words\times P(\bbbn)\to\bbbn]$,
there is a least $K_F$ up to an additive constant:
\begin{eqnarray*}
&\exists F\in \minpc[P(\bbbn)\times\words\to\bbbn]\ \
\forall G\in \minpc[P(\bbbn)\times\words\to\bbbn]\ \
\ {\cal K}_F\leqct {\cal K}_G&
\\
&\exists F\in \maxpc[P(\bbbn)\times\words\to\bbbn]\ \
\forall G\in \maxpc[P(\bbbn)\times\words\to\bbbn]\ \
\ {\cal K}_F\leqct {\cal K}_G&
\end{eqnarray*}
Such an $F$ is said to be optimal in
$\minpc[P(\bbbn)\times\words\to\bbbn]$ or in
$\maxpc[P(\bbbn)\times\words\to\bbbn]$.
\\
We let $\kmin[](\ ||\ )={\cal K}_F$ and
$\kmax[](\ ||\ )={\cal K}_F$
be some fixed such optimal functionals.
\medskip\\
{\bf 2.}
Let $(F_k)_{k\in\bbbn}$ be an enumeration of
$\minpc[P(\bbbn)\times\words\to\bbbn]$ which is itself in
$\minpc[\bbbn\times P(\bbbn)\times\words\to\bbbn]$.
Let ${\cal U}_{\min}$ be such that
$$\calUmin[](0^k1\ttp,A)=F_k(\ttp,A)\ \ \ \ \ \ \
\calUmin[](0^k)=0$$
Then $\calUmin[]$ is optimal in $\minpc[\WN]$.
\\
One defines similarly $\calUmax[]$ which is optimal in $\maxpc[]$.
\end{theorem}
\begin{remark}\label{rk:minmax}$\\ $
{\bf 1.}
Using the technique of \cite{ferbusgrigoKmaxKmin},
we see that there exists optimal functionals for $\maxpc[]$
of the form $\max F$ where
$F:\words\times\bbbn\times P(\bbbn)\to\bbbn$
is {\em total recursive}.
This is false for $\minpc[]$.
\medskip\\
{\bf 2.}
The inclusions $\pc[]\subseteq\maxpc[]\cap\minpc[]$ imply
that $\calkmin[](\ ||\ )\leqct {\cal K}(\ ||\ )$ and
$\calkmax[](\ ||\ )\leqct {\cal K}(\ ||\ )$.
Also, as is well-known, ${\cal K}(\ ||\ )\leqct\log$.
We can choose $\calkmin[],\calkmax[]$ so that the constant is $0$,
i.e. for all $x$ and $A$,
$$\calkmin[](x\,||\,A)\leq {\cal K}(x\,||\,A)\leq\log\ \ \ ,
\ \ \ \calkmax[](x\,||\,A)\leq {\cal K}(x\,||\,A)\leq\log$$
{\bf 3.}
In fact, the Min/Max hierarchy Theorem \ref{thm:hierarchyInfct}
extends to the functional setting. In \S\ref{ss:hierarchy} we shall
prove a much stronger result, cf. Thm.\ref{thm:hierarchy}.
\end{remark}
%
%
\section{Functional versus oracular}
\label{s:unifVersusOracle}
Functional Kolmogorov complexities allow for a uniform choice
of oracular Kolmogorov complexities.
The benefit of such a uniform choice was developed in
\S\ref{ss:functionalK} and is illustrated in the hierarchy theorem
in \S\ref{ss:hierarchy}.
\begin{theorem}\label{thm:unifVersusOracle}
Denote $K^A,\kmin[A],\kmax[A]:\bbbx\to\bbbn$ the Kolmogorov
complexities associated to the families $PR^A$ of partial
$A$-recursive functions and the families $\minprA[],\maxprA[]$
obtained by application of the $\min$ and $\max$ operators
to $PR^{A,\bbbx\times\bbbn\to\bbbn}$.
For all $A\subseteq\bbbn$, we have
$$K^A\eqct {\cal K}(\ ||\,A)\ \ ,\ \ \
\kmin[A]\eqct \calkmin[](\ ||\,A)\ \ ,\ \ \
\kmax[A]\eqct \calkmax[](\ ||\,A)$$
i.e.
$\forall\ A\in P(\bbbn)\ \exists c_A\ \forall\ttx\
\left\{\begin{array}{lcl}
|K^A(\ttx)-{\cal K}(\ttx\,||\,A)|&\leq&c_A\\
|\kmin[A](\ttx)-\calkmin[](\ttx\,||\,A)|)&\leq&c_A\\
|\kmax[A](\ttx)-\calkmax[](\ttx\,||\,A)|)&\leq&c_A
\end{array}\right.$
\end{theorem}

\begin{proof}
1. We let $(F_k)_{k\in\bbbn}$ and ${\cal U}$ be as in Point 2
of Thm.\ref{thm:uniformInvariance} and let
$F^A_k(\ttx)=F_k(\ttx,A)$ and $U^A(\ttx)={\cal U}(\ttx,A)$.
The sequence $(F^A_k)_{k\in\bbbn}$ is an enumeration of
the family $PR^{A,\words\to\bbbn}$ of partial $A$-recursive
functions $\words\to\bbbx$, which is partial $A$-recursive
as a function $\bbbn\times\words\to\bbbx$.
Since $U^A(0^k1\ttp)=F_k^A(\ttp)$, the classical
invariance theorem, in its relativized version, insures that
$U^A$ is optimal in $PR^{A,\words\to\bbbn}$,
whence $K^A\eqct K_{U^A}$.
\\
Now, ${\cal U}$ is optimal in $PC^{P(\bbbn)\times\words\to\bbbn}$,
whence ${\cal K}(\ ||\ )\eqct {\cal K}_{\cal U}(\ ||\ )$.
\\
Prop.\ref{p:uniformCond} insures
$K_{U^A}(\ttx)={\cal K}_{\cal U}(\ttx||A)$,
whence $K^A\eqct {\cal K}(\ ||A)$.
\medskip\\
2. The $Min$ and $Max$ cases are similar.
\end{proof}
%
%
\section{Refining the oracular Min/Max hierarchy with the
$\ll,\ll_\uparrow$ orderings}
\label{s:barzdins}
%
\subsection{Barzdins' theorem in a uniform setting}
\label{ss:barzdins}
%
The next lemma is essentially Barzdins' result cited in
\S\ref{ss:immune}.
In order (point 2) to get a relativized result with $\theta$
recursive rather than merely $A$-recursive, we shall look at the
oracle $A$ as a parameter and use uniform Kolmogorov complexity,
cf. \S\ref{ss:uniformInvariance}, \S\ref{s:unifVersusOracle}.
\begin{lemma}\label{l:ggK}$\\ $
{\bf 1.}
If $\ \varphi:\bbbn\to\bbbn$ is total recursive and tends to
$+\infty$ then $\{x : K(x)<\varphi(x)\}$ is an r.e. set which is
constructively $\Sigma^0_1$-dense.
\\
Moreover, this result is uniform in $\varphi$.
In fact, let $(\varphi_i)_{i\in\bbbn}$ and $(W_i)_{i\in\bbbn}$
be acceptable enumerations of partial recursive functions
$\bbbn\to\bbbn$ and r.e. subsets of $\bbbn$,
there are total recursive functions
$\xi:\bbbn\to\bbbn$ and $\theta:\bbbn^2\to\bbbn$ such that
\begin{enumerate}
\item[i.]
$\forall i\ \{x\in domain(\varphi_i):K(x)<\varphi_i(x)\}
=W_{\xi(i)}$
\item[ii.]
$\forall i,j\ (\varphi_i\mbox{ is unbounded on }
domain(\varphi_i)\cap W_j$

\hfill{$\Rightarrow\ (W_{\theta(i,j)}\mbox{ is infinite }\wedge\
W_{\theta(i,j)}\subseteq W_j\cap\{x:K(x)<\varphi_i(x)\}))$}
\end{enumerate}
%
{\bf 2.}
Consider second order Kolmogorov complexity ${\cal K}(x\,||\,A)$
and an acceptable enumeration $(\Phi_i)_{i\in\bbbn}$ of
partial computable functionals $\bbbn\times P(\bbbn)\to\bbbn$.
Using Thm.\ref{thm:unifVersusOracle} and Prop.\ref{p:uniform},
we shall consider ${\cal K}(x\,||\,A)$ as uniform Kolmogorov
relativization $K^A(x)$ and $\Phi_i(\ttx,A)$ as a uniform
oracle $A$ partial recursive function $\varphi^A_i(x)$.
We also denote $W^A_i=domain(\varphi^A_i)$.
\\
Point 1 relativizes uniformly, i.e., the above total
recursive functions
$\xi:\bbbn\to\bbbn$ and $\theta:\bbbn^2\to\bbbn$ can be taken so as
to satisfy {\em all} possible relativized conditions, i.e.
\begin{enumerate}
\item[i.]
$\forall i\ \forall A\
\{x\in domain(\varphi^A_i):K^A(x)<\varphi^A_i(x)\}
=W^A_{\xi(i)}$
\item[ii.]
$\forall i,j\ \forall A\ (\varphi^A_i\mbox{ is unbounded on }
domain(\varphi^A_i)\cap W^A_j$

\hfill{$\Rightarrow\ (W^A_{\theta(i,j)}\mbox{ is infinite }\wedge\
W^A_{\theta(i,j)}\subseteq W^A_j\cap\{x:K^A(x)<\varphi^A_i(x)\}))$}
\end{enumerate}
\end{lemma}
\begin{note}\label{note:best}
Lemma \ref{l:ggK} is optimal in the sense that there is no
possible $(\Pi^0_1,{\cal E})$-density result for
$\{x : K(x)<\varphi(x)\}$ since this set has $\Pi^0_1$ complement.
\end{note}
\begin{proof}
{\em Point 1i.}
Let $K=K_U$ where $U\in PR^{\words\to\bbbn}$. Then
\begin{eqnarray*}
K(x)<\varphi_i(x)&\Leftrightarrow&
\exists \ttp\ (|\ttp|<\varphi_i(x)\ \wedge\ U(\ttp)=x)
\end{eqnarray*}
which is a $\Sigma^0_1$ condition.
Therefore $\{(i,x):K(x)<\varphi_i(x)\}$ is r.e. and the
parametrization theorem yields the desired total recursive
function $\xi$.
\medskip\\
{\em Point 1ii.}\\
In order to prove constructive $\Sigma^0_1$-density
uniformly in $\varphi$, we first define a partial recursive
function $\alpha:\bbbn^2\times\words\to\bbbn$ such that
\begin{quote}
if there exists some $x\in W_j$ such that
$\varphi_i(x)\geq 2\,|\ttp|$ then $\alpha(i,j,\ttp)$
is such an $x$, else $\alpha(i,j,\ttp)$ is undefined.
\end{quote}
Then we shall use the facts that
$$K\leqct K_{\ttp\mapsto\alpha(i,j,\ttp)}\ \ ,\ \
K_{\ttp\mapsto\alpha(i,j,\ttp)}(\alpha(i,j,\ttp))\leqct |\ttp|$$
to get an inequality $K(\alpha(i,j,\ttp))\leqct|\ttp|$ from which
$K(\alpha(i,j,\ttp))<\varphi_i(\ttp)$ can be deduced.
\medskip\\
{\em a.}
The formal definition of $\alpha$ as a partial recursive function
is as follows.
\\
Denote $W_{j,t}$ the finite subset of $W_j$ obtained
after $t$ steps of its standard enumeration.
Let $Z_t:\bbbn^2\times\words\to P(\bbbn)$ be such that
$$Z_t(i,j,\ttp)=\{x\in W_{j,t} : \varphi_i(x)
\mbox{ halts in }\leq t\mbox{ steps and is }> 2\,|\ttp|\}$$
Clearly, $\{(t,i,j,\ttp):Z_t(i,j,\ttp)\neq\emptyset\}$ is a
recursive subset of $\bbbn^3\times\words$.
Thus, we can define the partial recursive function $\ \alpha$
as follows:
\begin{eqnarray*}
domain(\alpha)&=&\{(i,j,\ttp):\exists t\ Z_t(i,j,\ttp)\neq\emptyset\}
\\
\alpha(i,j,\ttp)
&=&\mbox{the first element in }Z_t(i,j,\ttp)
\\ &&\mbox{where }t\mbox{ is least such that }
                Z_t(i,j,\ttp)\neq\emptyset
\end{eqnarray*}
Let $(\psi_i)_{i\in\bbbn}$ be an acceptable enumeration of
$PR^{\words\to\bbbn}$. Since $\alpha$ is partial recursive,
there exists a total recursive function $\eta:\bbbn^2\to\bbbn$ such
that $\alpha(i,j,\ttp)=\psi_{\eta(i,j)}(\ttp)$ for all $i,j,\ttp$.
Finally, we let $\theta$ be a total recursive function such that
$$W_{\theta(i,j)}=\psi_{\eta(i,j)}(\{\ttp : |\ttp|>\eta(i,j)
\ \wedge\ (i,j,\ttp)\in domain(\psi_{\eta(i,j)})\})$$
Since $\alpha$ and $\psi_{\eta(i,j)}$ take values
in $W_j$, we have $W_{\theta(i,j)}\subseteq W_j$ for all $i,j$.
\medskip\\
{\em b.} Let $U:\words\to\bbbn$ be such that
$U(0^k1\ttp)=\psi_k(\ttp)$ and $U(0^k)=\psi_k(\lambda)$
(where $\lambda$ is the empty word).
The usual invariance theorem insures that $U$ is optimal.
Thus, we can (and shall) suppose that $K=K_U$.
\\
Since $\alpha(i,j,\ttp)=\psi_{\eta(i,j)}(\ttp)$, we have
$\alpha(i,j,\ttp)=U(0^{\eta(i,j)}1\ttp)$.
Thus, for any $(i,j,\ttp)\in domain(\alpha)$,
$$K(\alpha(i,j,\ttp))=K_U(U(0^{\eta(i,j)}1\ttp))
\leq|0^{\eta(i,j)}1\ttp|=|\ttp|+\eta(i,j)+1$$
{\em c.} Suppose now that $\varphi_i$ is unbounded on
$domain(\varphi_i)\cap W_j$.
Then, for all $\ttp$, the set
$Z_t(i,j,\ttp)$ is non empty for $t$ big enough,
so that $\alpha(i,j,\ttp)=\psi_{\eta(i,j)}(\ttp)$ is defined
for all $\ttp$.
Also, due to the definition of $Z_t$, we see that
$\alpha(i,j,\ttp)$ tends to $+\infty$ with the length of $\ttp$.
In particular, $W_{\theta(i,j)}$ is infinite.
\\
 From the definition of $\alpha$, we see that
$\varphi_i(\alpha(i,j,\ttp))> 2|\ttp|$.
Using {\em b}, we see that for all $|\ttp|>\eta(i,j)$, we have
$K(\alpha(i,j,\ttp))\leq2\,\ttp<\varphi(\alpha(i,j,\ttp))$.
\\
This proves that $W_{\theta(i,j)}$ is included in
$\{x : K(x)<\varphi_i(x)\}$.
\\
Thus, $W_{\theta(i,j)}$ is an infinite r.e. set included in
$W_j \cap\ \{x : K(x)<\varphi_i(x)\}$.
\medskip\\
{\em Point 2i.}
Let ${\cal K}(\ ||\ )={\cal K}_{\cal U}$ where
${\cal U}\in PC^{P(\bbbn)\times\words\to\bbbn}$. Then
\begin{eqnarray*}
K^A(x)<\varphi^A_i(x)&\Leftrightarrow&
\exists \ttp\ (|\ttp|<\varphi^A_ix\ \wedge\ {\cal U}(\ttp,A)=x)
\end{eqnarray*}
which is a $\Sigma^0_1$ condition.
Therefore $\{(i,x,A):K^A(x)<\varphi^A_i(x)\}$ is $\Sigma^0_1$
and the parametrization property (cf. Def.\ref{def:acceptable})
yields the desired total recursive function $\xi$.
\medskip\\
{\em Point 2ii.}
The proof is similar to that of Point 1ii. Just add everywhere a
second order argument $A$ varying in $P(\bbbn)$
and use the parametrization property of Def.\ref{def:acceptable}.
Thus, $\alpha$ is now a partial computable functional
\\\centerline{$\alpha:\bbbn^2\times\words\times P(\bbbn)\to\bbbn$}
The enumeration $(\psi_i)_{i\in\bbbn}$ now becomes an enumeration
$(\Psi_i)_{i\in\bbbn}$ of the partial computable functionals
$\words\times P(\bbbn)\to\bbbn$.
The total recursive functions $\eta,\theta$ are now such that
$\ \alpha(i,j,\ttp,A)=\Psi_{\eta(i,j)}(\ttp,A)$ and
$$W^A_{\theta(i,j)}=\{\Psi_{\eta(i,j)}(\ttp,A) :
\ttp\mbox{ such that }(i,j,\ttp,A)\in domain(\alpha)\
\wedge\ |\ttp|>\eta(i,j)\}$$
The arguments in {\em b,c} above go through with the superscript
$A$ everywhere and with ${\cal U}$ (cf. proof of Point 2i above)
in place of $U$.
\end{proof}
\begin{remark}$\\ $
{\bf 1.}
Lemma \ref{l:ggK} still holds for $\phi\in\maxpr[]$
in place of $\varphi\in PR$.
However, this does not really add: an easy argument shows
that if $\phi\in\maxpr[]$ and $W_j\subseteq domain(\phi)$ is
infinite then there exists an infinite $W_k\subseteq W_j$
and $\varphi_i\in PR$ such that $W_k\subseteq domain(\varphi_i)$ and
$\varphi_i(x)\leq\phi(x)$ for all
$x\in domain(\varphi_i)$.
Moreover, $k$ and $i$ can be given by total recursive functions
depending on $j$ and a code for $\phi$ in $\maxpr[]$.
\\
This also holds uniformly: replace $\varphi$ by a functional
$\Phi\in\pc[]$.
\medskip\\
{\bf 2.}
Of course, Lemma \ref{l:ggK}  cannot hold for
$\phi\in\minpr[]$ since $K$ is itself in $\minpr[]$.
\end{remark}
%
%
\subsection{Comparing $K$ and $\kmax[]$ \`a la Barzdins}
\label{ss:KKmaxalabarzdins}
%
In this subsection and the next one, we now come to central
results of the paper, namely,
{\em \\\indent - $K$ can be compared to $\kmax[],\kmin[]$ via
the $\lless[]\,$ and $\lless[]\uparrow$ orderings,
\\\indent - $\kmax[],\kmin[]$ can be compared via the
$\often[]{\uparrow}$ relation.}
\medskip\\
\begin{notation}\label{not:cupdense}
We shall write $X$ is $({\cal C}_1\cup{\cal C}_2,{\cal D})$-dense
to mean $X$ is $({\cal C}_1,{\cal D})$-dense and
$({\cal C}_2,{\cal D})$-dense.
\end{notation}
\begin{remark}\label{rk:cupdense}
Let ${\cal C}_1\vee{\cal C}_2$ be the family of sets
$R_1\cup R_2$ where $R_1\in{\cal C}_1$ and $R_2\in{\cal C}_2$.
\\
If ${\cal C}_1,{\cal C}_2$ both contain the empty set
(which is usually the case), then
${\cal C}_1\cup{\cal C}_2\subseteq{\cal C}_1\vee{\cal C}_2$,
and therefore $({\cal C}_1\vee{\cal C}_2,{\cal D})$-density
(resp. constructive density) always implies
$({\cal C}_1\cup{\cal C}_2,{\cal D})$-density
(resp. constructive density).
\\
Conversely, every infinite set in ${\cal C}_1\vee{\cal C}_2$
contains an infinite subset in ${\cal C}_1$ or in ${\cal C}_2$,
so that $({\cal C}_1\cup{\cal C}_2,{\cal D})$-density implies
--- hence is equivalent to ---
$({\cal C}_1\vee{\cal C}_2,{\cal D})$-density. However, this is
no more true as concerns {\em constructive density}: if
$R_1\cup R_2$ is infinite one cannot decide (from codes)
which one of $R_1$ and $R_2$ is infinite.
\end{remark}
\begin{lemma}\label{l:ggKmax}$\\ $
{\bf 1.}
Suppose $\phi:\bbbn\to\bbbn$ is a total function in $\minpr[]$
which is monotone and tends to $+\infty$.
Then the set $\{x:\kmax[](x)<\phi(K(x))\}$ is constructively
$\ (\Sigma^0_1\cup\Pi^0_1,
\exists^{\leq\phi}(\Sigma^0_1\wedge\Pi^0_1))$-dense
(cf. Def.\ref{def:DEDense} Point 3).
\\
Moreover, this result is uniform in $\phi$.
In fact, let $(\phi_i)_{i\in\bbbn}$ and $(W_i)_{i\in\bbbn}$ be
acceptable enumerations of $\minpr[]$ and r.e. subsets of $\bbbn$.
There are total recursive functions
$\theta_0,\theta_1:\bbbn^2\to\bbbn$ such that, for all $i,j,k$,
with the notations of Def.\ref{def:Dclass},
if $\phi_i\in\minpr[]$ is total, monotone and tends to $+\infty$
then
\begin{eqnarray*}
W_j\mbox{ is infinite}&\Rightarrow&
W^{\truc}_{\theta_0(i,j)}\mbox{ is an infinite subset of}\\
&&W_j\cap\{x:\kmax[](x)<\phi(K(x))\}
\\
\bbbn\setminus W_k\mbox{ is infinite}&\Rightarrow&
W^{\truc}_{\theta_1(i,k)}\mbox{ is an infinite subset of}\\
&&(\bbbn\setminus W_k)\cap\{x:\kmax[](x)<\phi(K(x))\}
\end{eqnarray*}
{\bf 2.}
Consider Kolmogorov relativizations $K^A,\kmax[A]$ obtained
from second order Kolmogorov complexities ${\cal K},\calkmax[]$
(cf. Thm.\ref{thm:unifVersusOracle})
and enumerations $(\phi^A_i)_{i\in\bbbn}$ and $(W^A_i)_{i\in\bbbn}$
of $\minpr[A]$ and $A$-r.e. sets which come from acceptable
enumerations of functionals in
$\minpc[\bbbn\times P(\bbbn)\to\bbbn]$ and of
$\Sigma^0_1$ subsets of $\bbbn\times P(\bbbn)$
(cf. Prop.\ref{p:uniform}).
We shall also use notations from Def.\ref{def:Dclass}.
\\
Point 1 relativizes uniformly, i.e., the above total
recursive functions $\theta_0,\theta_1:\bbbn^2\to\bbbn$ can be
taken so as to satisfy {\em all} possible relativized conditions.
I.e., if $\phi^A_i\in\minpr[A]$ is total, monotone and tends to
$+\infty$ then
\begin{eqnarray*}
W^A_j\mbox{ is infinite}&\Rightarrow&
W^{\trucA}_{\theta_0(i,j)})\mbox{ is an infinite subset of}\\
&&W^A_j\cap\{x:\kmax[A](x)<\phi^A(K^A(x))\}
\\
\bbbn\setminus W^A_k\mbox{ is infinite}&\Rightarrow&
W^{\trucA}_{\theta_1(i,k)}\mbox{ is an infinite subset of}\\
&&(\bbbn\setminus W^A_k)\cap\{x:\kmax[A](x)<\phi^A(K^A(x))\}
\end{eqnarray*}
\end{lemma}
\begin{proof}
{\em 1. The strategy.}\\
We essentially keep the strategy of the proof of Lemma
\ref{l:ggK}.
The idea is, for given $i,j$, to construct a $\maxpr[]$ function
$\alpha:\bbbn^2\times\words$ such that $\alpha(i,j,\ttp)$ is in
$W_j$ (or in $\bbbn\setminus W_j$) and
$\varphi_i(K(\alpha(i,j,\ttp)))>2\,|\ttp|$.
Then to use inequalities
$$\kmax[]\leqct K_{\ttp\mapsto\alpha(i,j,\ttp)}\ \ ,\ \
K_{\ttp\mapsto\alpha(i,j,\ttp)}(\alpha(i,j,\ttp))\leqct |\ttp|$$
to get an inequality $\kmax[](\alpha(i,j,\ttp))\leqct|\ttp|$
from which
$\kmax[](\alpha(i,j,\ttp))<\varphi_i(K(\alpha(i,j,\ttp)))$
can be deduced.
\medskip\\
As we have to deal with $\Sigma^0_1$ sets and with $\Pi^0_1$ sets,
i.e. sets of the form $W_j$ or $\bbbn\setminus W_k$, we shall define
two such functions $\alpha$, namely $\alpha_0,\alpha_1$.
\\
In order to get these functions in $\maxpr[]$, we  define
partial recursive functions
$a_0,a_1:\bbbn^2\times\words\times\bbbn\to\bbbn$
and set $\max a_0=\alpha_0$ and $\max a_1=\alpha_1$.
\medskip\\
{\em 2. Approximation of $\phi_i$ from above.}\\
Let $\phi_i(x)=\min_t\varphi_i(x,t)$ where
$(\varphi_i)_{i\in\bbbn}$ is an acceptable enumeration of
$PR^{\bbbn\times\bbbn\to\bbbn}$.
\\
Using the parametrization theorem, let $\xi:\bbbn\to\bbbn$ be a
total recursive function such that $\varphi_{\xi(i)}$ has domain
$\{x:\exists u\ \varphi_i(x,t)\mbox{ does halt}\}\times\bbbn$
and satisfies
\begin{eqnarray*}
\varphi_{\xi(i)}(x,0)
&=&\varphi_i(x,u)\mbox{ where $u$ is least such that }
                \varphi_i(x,u) \mbox{ does halt}
\\
\varphi_{\xi(i)}(x,t+1)
&=&\min(\{\varphi_{\xi(i)}(x,0)\}\cup
\{\varphi_i(x,v): v\leq t\ \mbox{and}\\
&&\hspace{4cm}\varphi_i(x,v)\mbox{ halts in $\leq t$ steps}\})
\end{eqnarray*}
Observe that $\varphi_{\xi(i)}(x,t)$ is decreasing in $t$ and
$\phi_i(x)=\min_t\varphi_{\xi(i)}(x,t)$, so that
$\varphi_{\xi(i)}(x,t)$ is a partial recursive approximation of
$\phi_i(x)$ from above.
\\
Also, for any given $i,x$,
either $\phi_i(x)$ is undefined and $\varphi_{\xi(i)}(x,t)$
is defined for no $t$
or $\phi_i(x)$ is defined and $\varphi_{\xi(i)}(x,t)$
is defined for all $t$.
\medskip\\
{\em 3. Functions $a_\epsilon$ and $\alpha_\epsilon$.}\\
Denote $W_{j,t}$ the finite subset of $W_j$ obtained
after $t$ steps of its standard enumeration.
Denote $K^t(x)$ some total, recursive approximation of $K(x)$
from above which is decreasing in $t$ (cf. Prop.\ref{p:approxK}).
\\
We define $a_0,a_1$ as follows:
\begin{eqnarray*}
a_0(i,j,\ttp,0)
&=&\mbox{the element which appears first in the standard}\\
&&\mbox{ enumeration of $W_j$ (hence undefined if $W_j$ is empty)}
\\
&&\mbox{}\\
a_0(i,j,\ttp,t+1)
&=&\left\{
\begin{array}{ll}
a_0(i,j,\ttp,t)
&\mbox{if }\varphi_{\xi(i)}(K^t(a_0(i,j,\ttp,t)),t)>2\,|\ttp|
\medskip\\
x&\mbox{if }\varphi_{\xi(i)}(K^t(a_0(i,j,\ttp,t)),t)\leq2\,|\ttp|\\
&\mbox{and $x$ is the next element which}\\
&\mbox{appears in the standard enumeration}\\
&\mbox{of $W_i$ and satisfies }x>a_0(i,j,\ttp,t)
\\
undefined&\mbox{if }\varphi_{\xi(i)}(K^t(a_0(i,j,\ttp,t)),t)
\mbox{ is undefined}
\end{array}\right.
\end{eqnarray*}
and
\begin{eqnarray*}
a_1(i,k,\ttp,0)&=&0
\medskip\\
a_1(i,k,\ttp,t+1)
&=&\left\{
\begin{array}{ll}
a_1(i,k,\ttp,t)
&\mbox{if }\varphi_{\xi(i)}(K^t(a_1(i,j,\ttp,t)),t)>2\,|\ttp|\\
&\mbox{and }a_1(i,k,\ttp,t)\notin W_{k,t}
\medskip\\
a_1(i,k,\ttp,t)+1
&\mbox{if }\varphi_{\xi(i)}(K^t(a_1(i,j,\ttp,t)),t)\leq2\,|\ttp|\\
&\mbox{or }(\varphi_{\xi(i)}(K^t(a_1(i,j,\ttp,t)),t)>2\,|\ttp|\\
&\mbox{and }a_1(i,k,\ttp,t)\in W_{k,t})
\\
undefined&\mbox{if }\varphi_{\xi(i)}(K^t(a_1(i,k,\ttp,t)),t)
\mbox{ is undefined}
\end{array}\right.
\\ &&\mbox{}
\end{eqnarray*}
{\bf Claim. }{\em
Suppose $\phi_i\in\minpr[]$ is total monotone increasing and
tends to $+\infty$.}
\\
{\bf a.\ }{\em If $W_j$ is infinite then
$(\ttp,t)\mapsto a_0(i,j,\ttp,t)$ and
$\ttp\mapsto\alpha_0(i,j,\ttp)$ are total functions and
$$\forall \ttp\ (\alpha_0(i,j,\ttp)\in W_j\
\wedge\ \phi_i(K(\alpha_0(i,j,\ttp)))>2\,|\ttp|)$$}
{\bf b.\ }{\em Function $a_1$ is always total.
If $\bbbn\setminus W_k$ is infinite then
$\ttp\mapsto\alpha_1(i,k,\ttp)$ is a total function and
$$\forall \ttp\ (\alpha_1(i,k,\ttp)\notin W_k\
\wedge\ \phi_i(K(\alpha_1(i,k,\ttp)))>2\,|\ttp|)$$}
{\em Proof of Claim. }\\
As seen in 2 above, if $\phi_i$ is total so is $\varphi_{\xi(i)}$.
This insures the total character of $a_0$ (resp. $a_1$).
\\
Fix some $\ttp$. Since $\varphi_{\xi(i)}(x,t)\geq\phi_i(x)$,
$\phi_i$ is monotone increasing and $K,\phi_i$ are total and
tend to $+\infty$, for all large enough $x$ and all $t$, we have
$$\varphi_{\xi(i)}(K^t(x),t)\geq\phi_i(K^t(x))
\geq\phi_i(K(x))>2\,|\ttp|$$
Suppose $W_j$ is infinite. Then there are elements in $W_j$
which satisfy $\phi_i(K(x))>2\,|\ttp|$.
Let $x_0(i,j,\ttp)$ be such an element which appears first in the
standard enumeration of $W_j$.
It is easy to see that, for all $t$ large enough, we have
$a_0(i,j,\ttp,t)=x_0(i,j,\ttp)$.
Thus, $\alpha_0(i,j,\ttp)=x_0(i,j,\ttp)$ is defined and
$\alpha_0(i,j,\ttp)\in W_j\cap\{\phi_i(K(x))>2\,|\ttp|\}$.
Which proves Point a of the Claim.
\medskip\\
Suppose $\bbbn\setminus W_k$ is infinite. Then there are elements
in $\bbbn\setminus W_k$ which satisfy $\phi_i(K(x))>2\,|\ttp|$.
Let $x_1(i,k,\ttp)$ be the least such element.
It is easy to see that, for all $t$ large enough (namely, for $t$
such that $W_k\cap[0,x_1(i,k,\ttp)[\subseteq W_{k,t}$), we have
$a_1(i,k,\ttp,t)=x_1(i,k,\ttp)$.
Thus, $\alpha_1(i,k,\ttp)=x_1(i,k,\ttp)$ is defined and
$\alpha_1(i,k,\ttp)\in
(\bbbn\setminus W_k)\cap\{\phi_i(K(x))>2\,|\ttp|\}$.
Which proves Point b of the Claim.\hfill{$\Box$ \em (Claim)}
\medskip\\
{\em 4. Functions $\eta_\epsilon,\theta_\epsilon$.}\\
Let $(\psi_n)_{n\in\bbbn}$ be an acceptable enumeration of
partial recursive functions $\words\times\bbbn\to\bbbn$.
Since $a_0,a_1:\bbbn^2\times\words\times\bbbn\to\bbbn$ are partial
recursive, the parametrization property insures that there exists
total recursive functions $\eta_0,\eta_1:\bbbn^2\to\bbbn$ such that,
for all $i,j,k,\ttp,t$ and $\epsilon=0,1$,
$$a_\epsilon(i,j,\ttp,t)=\psi_{\eta_\epsilon(i,j)}(\ttp,t)$$
Taking the $\max$ over $t$, and letting
$\alpha_\epsilon=\max a_\epsilon$, we get, for all $i,j,\ttp$,
$$\alpha_\epsilon(i,j,\ttp,t)
=(\max\psi_{\eta_\epsilon(i,j)})(\ttp)$$
For all $i,j,k$, set
$$Y_\epsilon(i,j)=\{\alpha_\epsilon(i,j,\ttp) :
(i,j,\ttp)\in domain(\alpha_\epsilon)\
\wedge\ |\ttp|>\eta_\epsilon(i,j)\}$$
Using the Claim and inequality $K(y)\leq y$
(which we always can suppose), observe that
\begin{eqnarray*}
y\in Y_\epsilon(i,j)&\Leftrightarrow&
\exists\ttp\ (2\,|\ttp|<\phi_i(K(y))\
\wedge\ |\ttp|>\eta_\epsilon(i,j)
\ \wedge y=\alpha_\epsilon(i,j,\ttp))
\\
&\Leftrightarrow&
\exists\ttp\ (|\ttp|<\phi_i(y)\
\wedge\ |\ttp|>\eta_\epsilon(i,j)
\ \wedge y=\alpha_\epsilon(i,j,\ttp))
\end{eqnarray*}

Using Prop.\ref{p:syntax}, we see that this is $\truc$ in
$i,j,y$ (cf. Def.\ref{def:existsleq}).
We let $\theta_0,\theta_1:\bbbn^2\to\bbbn$ be total recursive
functions such that
$$W^{\truc[\bbbn]}_{\theta_\epsilon(i,j)}=Y_\epsilon(i,j)$$
{\em 5. Point 1 of the Lemma.}\\
Let $U:\words\times\bbbn\to\bbbn$ be such that
$U(0^n1\ttp,t)=\psi_n(\ttp,t)$ and $U(0^n,t)=\psi_n(\lambda,t)$
(where $\lambda$ is the empty word).
Taking the $\max$ over $t$, we get
$(\max U)(0^n1\ttp)=(\max\psi_n)(\ttp)$ and
$(\max U)(0^n)=(\max\psi_n)(\lambda)$.
Since the $\max\psi_n$'s enumerate $\maxpr[\words\to\bbbn]$,
the invariance theorem \ref{thm:invarianceMinMax} insures that
$\max U$ is optimal in $\maxpr[\words\to\bbbn]$.
Thus, we can (and shall) suppose that $\kmax[]=K_{(\max U)}$.
\\
Since $\alpha_\epsilon(i,j,\ttp)
=(\max\psi_{\eta_\epsilon(i,j)})(\ttp)
=(\max U)(0^{\eta_\epsilon(i,j)}1\ttp)$, we get
\begin{eqnarray*}
\kmax[](\alpha_\epsilon(i,j,\ttp))
&=&K_{(\max U)}((\max U)(0^{\eta_\epsilon(i,j)}1\ttp)\\
&\leq&\eta_\epsilon(i,j)+1+|\ttp|\\
&\leq&2\,|\ttp|\hspace{1cm}\mbox{in case }|\ttp|>\eta_\epsilon(i,j)
\end{eqnarray*}
Suppose $\phi_i$ is total, monotone and tends to $+\infty$ and
$W_j$ (resp. $\bbbn\setminus W_k)$ is infinite.
Using the last inequality and that from the above Claim relative
to $\epsilon=0$ (resp. $\epsilon=1$),
we see that, for $|\ttp|>\eta_\epsilon(i,j)$, we have
$$K_{(\max U)}(\alpha_\epsilon(i,j,\ttp))
\leq2\,|\ttp|<\phi_i(K(\alpha_\epsilon(i,j,\ttp))$$
Which proves that $W^{\truc[\bbbn]}_{\theta_\epsilon(i,j,k)}$
is included in $\{x:\kmax[](x)<\phi_i(K(x))\}$.
Using the Claim again, this set is also included in
$W_j$ (resp. $(\bbbn\setminus W_k)$.
This finishes the proof of Point 1 of the Lemma.
\medskip\\
{\em 6. Point 2 of the Lemma.}\\
The proof is similar to that of Point 1. Just add everywhere a
second order argument $A$ varying in $P(\bbbn)$
and use the parametrization property of Def.\ref{def:acceptable}.
Thus, $a_0,a_1$ are now partial computable functionals
\\\centerline{$\bbbn^2\times\words\times P(\bbbn)\times\bbbn
\to\bbbn$}
The enumeration $(\psi_n)_{n\in\bbbn}$ now becomes an enumeration
$(\Psi_n)_{n\in\bbbn}$ of the partial computable functionals
$\words\times P(\bbbn)\to\bbbn$.
The total recursive functions $\eta_\epsilon,\theta_\epsilon$ are
now such that
\begin{eqnarray*}
a_\epsilon(i,j,\ttp,A,t)&=&\Psi_{\eta_\epsilon(i,j)}(\ttp,A,t)
\\
\alpha_\epsilon(i,j,\ttp,A)
&=&(\max\Psi_{\eta_\epsilon(i,j)})(\ttp,A)
\\
W^A_{\theta_\epsilon(i,j,j)}&=&\{\alpha_\epsilon(i,j,\ttp,A):
(i,j,\ttp,A)\in domain(\alpha_\epsilon)\
\wedge\ |\ttp|>\eta_\epsilon(i,j)\}
\end{eqnarray*}
and $U$ has to be changed to
${\cal U}\in PC^{P(\bbbn)\times\words\to\bbbn}$ such that
${\cal K}(\ ||\ )={\cal K}_{\cal U}$.
\\
The arguments for the proof of Point 1 above go through with the
superscript $A$ everywhere.
\end{proof}
%
%
\subsection{Comparing $K$ and $\kmin[]$ \`a la Barzdins}
\label{ss:KKminalabarzdins}
%
We shall need the following notion to get an analog of Lemma
\ref{l:ggKmax} with $\kmin[]$.
\begin{definition}\label{def:recbounded}
The growth function of an infinite set $X\subseteq\bbbn$
is defined as $$growth_X(n)=(n+1)\mbox{-th point of }X$$
The infinite set $X$ has recursively bounded growth if
$growth_X\leq\psi$ for some total recursive function
$\psi:\bbbn\to\bbbn$.
\end{definition}
\begin{lemma}\label{l:ggKmin}$\\ $
{\bf 1.}
Let's denote $\widetilde{\Pi^{0,A}_1}$ the family of infinite
$\Pi^{0,A}_n$ subsets of $\bbbn$ with $A$-recursively bounded
growth.\\
Suppose $\phi:\bbbn\to\bbbn$ is a total function in $\minpr[]$
which is monotone and tends to $+\infty$.
Then the set $\{x:\kmin[](x)<\phi(K(x))\}$ is constructively
$\ (\Sigma^0_1\cup\widetilde{\Pi^0_1},
\exists^{\leq\phi}(\Sigma^0_1\wedge\Pi^0_1))$-dense
(cf. Def.\ref{def:DEDense} Point 3).
\\
Moreover, this result is uniform in $\phi$ and in a recursive
$\psi$ bound for the $\Pi^0_1$ set.
In fact, let $(\phi_i)_{i\in\bbbn}$, $(\psi_m)_{m\in\bbbn}$ and
$(W_i)_{i\in\bbbn}$ be acceptable enumerations of $\minpr[]$, $PR$
and of r.e. subsets of $\bbbn$.
There are total recursive functions
$\theta_0:\bbbn^2\to\bbbn$ and $\theta_1:\bbbn^3\to\bbbn$ such that,
for all $i,j,m,k$, with the notations of Def.\ref{def:Dclass},
if $\phi_i\in\minpr[]$ and $\psi_m\in PR$ are total, monotone
and tend to $+\infty$ then
\begin{eqnarray*}
W_j\mbox{ is infinite}&\Rightarrow&
W^{\truc}_{\theta_0(i,j)}\mbox{ is an infinite subset of}\\
&&W_j\cap\{x:\kmin[](x)<\phi_i(K(x))\}
\end{eqnarray*}
\indent$\bbbn\setminus W_k\mbox{ is infinite and }
\psi_m\geq growth_{\bbbn\setminus W_k}$
\medskip

\hspace{3cm}
$\Rightarrow\ W^{\truc}_{\theta_1(i,k,m)}$ is an infinite subset of
\medskip

\hspace{4cm}
$(\bbbn\setminus W_k)\cap\{x:\kmin[](x)<\phi_i(K(x))\}$
\medskip\\
{\bf 2.}
Consider Kolmogorov relativizations $K^A,\kmin[A]$ and enumerations
$(\phi^A_i)_{i\in\bbbn}$ and $(\psi^A_m)_{m\in\bbbn}$ and
$(W^A_i)_{i\in\bbbn}$ of $\minpr[A]$, $\PR[A]$
and $A$-r.e. sets  as in Point 2 of Lemma \ref{l:ggKmax}.
\\
Point 1 relativizes uniformly, i.e., the above total
recursive functions $\theta_0,\theta_1:\bbbn^2\to\bbbn$ can be
taken so as to satisfy {\em all} possible relativized conditions.
I.e., if $\phi^A_i,\psi^A_m$ are total, monotone and tend
to $+\infty$ then
\begin{eqnarray*}
W^A_j\mbox{ is infinite}&\Rightarrow&
W^{\trucA}_{\theta_0(i,j)}\mbox{ is an infinite subset of}\\
&&W^A_j\cap\{x:\kmin[A](x)<\phi^A_i(K^A(x))\}
\end{eqnarray*}
\indent$\bbbn\setminus W^A_k\mbox{ is infinite and }
\psi^A_m\geq growth_{\bbbn\setminus W^A_k}$
\medskip

\hspace{3cm}
$\Rightarrow\ W^{\trucA}_{\theta_1(i,k,m)}$ is an infinite subset of
\medskip

\hspace{4cm}
$(\bbbn\setminus W^A_k)\cap\{x:\kmin[A](x)<\phi^A(K^A(x))\}$
\end{lemma}
\begin{proof}
{\em 1. The strategy.}
The proof follows that of Lemma \ref{l:ggKmax} except that
now $\alpha_\epsilon$ is equal to $\min a_\epsilon$ and that
$a_1$ and $\alpha_1$ also depend on the index $m$ of the recursive
majorant $\psi_m$ of the growth function of the $\Pi^0_1$ set.
\\
Since $\alpha_0$ (resp. $\alpha_1$) has to be in $\minpr[]$, i.e.
is to be recursively approximated {\em from above}, we have to force
that, for given $i,j,k,m,\ttp$, the first defined
$a_0(i,j,\ttp,t)$ (resp. $a_1(i,k,m,\ttp,t)$) majorizes an element
$x$ of $W_j$ (resp. $\bbbn\setminus W_k$) which is such that
$\kmin[](x)<\phi_i(K(x))$.
\\
To insure this, we choose $a_0(i,j,\ttp,0)$
(resp. $a_1(i,k,m,\ttp,0)$) so that the interval
$[0,a_\epsilon(i,j,\ttp,0)[$ (resp. $[0,a_\epsilon(i,k,m,\ttp,0)[$)
contains at least $2^{2|\ttp|+1}$ points in
$\{x:\kmin[](x)<\phi_i(K(x))\}$.
\medskip\\
{\em 2.}
We shall use the partial recursive approximation from above
$\varphi_{\xi(i)}(x,t)$ of $\phi_i(x)$ defined in point 2 of the
proof of Lemma \ref{l:ggKmax}.
\medskip\\
{\em 3. Functions $a_\epsilon$ and $\alpha_\epsilon$.}\\
Let $Z_0(i,j,\ttp)$ be the set of $2^{2|\ttp|+1}$ distinct elements
which appear first in the standard enumeration of $W_i$.
We define $a_0$ as follows:
\begin{eqnarray*}
a_0(i,j,\ttp,0)&=&\mbox{the largest element of }Z_0(i,j,\ttp)
\medskip\\
a_0(i,j,\ttp,t+1)
&=&\left\{
\begin{array}{ll}
a_0(i,j,\ttp,t)
&\mbox{if }\varphi_{\xi(i)}(K^t(a_0(i,j,\ttp,t)),t)>2\,|\ttp|
\medskip\\
x&\mbox{if }\varphi_{\xi(i)}(K^t(a_0(i,j,\ttp,t)),t)\leq2\,|\ttp|\\
&\mbox{and $x$ is the largest element of}\\
&Z_0(i,j,\ttp)\ \cap\ [0,a_0(i,j,\ttp,t)[
\medskip\\
undefined&\mbox{if }\varphi_{\xi(i)}(K^t(a_0(i,j,\ttp,t)),t)
\mbox{ is undefined}
\end{array}\right.
\end{eqnarray*}
\medskip\\
We now define $a_1$, using the recursive majorant $\psi_m$.
\begin{eqnarray*}
a_1(i,k,m,\ttp,0)&=&\psi_m(2^{2|\ttp|+1})
\medskip\\
&&\mbox{Let }u=\varphi_{\xi(i)}(K^t(a_1(i,k,m,\ttp,t)),t)
\medskip\\
a_1(i,k,m,\ttp,t+1)
&=&\left\{
\begin{array}{ll}
a_1(i,k,m,\ttp,t)
&\mbox{if }u>2\,|\ttp|\\
&\mbox{and }a_1(i,k,m,\ttp,t)\notin W_{k,t}
\medskip\\
a_1(i,k,m,\ttp,t)-1
&\mbox{if }u\leq2\,|\ttp|\mbox{ or }(u>2\,|\ttp|\\
&\mbox{ and }a_1(i,k,m,\ttp,t)\in W_{k,t})
\medskip\\
undefined&\mbox{if }u\mbox{ is undefined}
\end{array}\right.\end{eqnarray*}
Clearly, $a_0$ and $a_1$ are partial recursive.
\medskip\\
{\bf Claim. }{\em
Suppose $\phi_i\in\minpr[]$ is total monotone increasing and
tends to $+\infty$.}
\\
{\bf a.\ }{\em If $W_j$ is infinite then
$(\ttp,t)\mapsto a_0(i,j,\ttp,t)$ and
$\ttp\mapsto\alpha_0(i,j,\ttp)$ are total functions and
$$\forall \ttp\ (\alpha_0(i,j,\ttp)\in W_j\
\wedge\ \phi_i(K(\alpha_0(i,j,\ttp)))>2\,|\ttp|)$$}
{\bf b.\ }{\em Function $a_1$ is total.
If $\bbbn\setminus W_k$ is infinite and $\psi_m$ is
a total recursive function such that
$\psi_m\geq growth_{\bbbn\setminus W_k}$
then $\ttp\mapsto\alpha_1(i,k,m,\ttp)$ is a total function and
$$\forall \ttp\ (\alpha_1(i,k,m,\ttp)\notin W_k\
\wedge\ \phi_i(K(\alpha_1(i,k,m,\ttp)))>2\,|\ttp|)$$}
{\em Proof of Claim. }\\
As seen in the proof of Lemma \ref{l:ggKmax}, if $\phi_i$ is total
then so are $\varphi_{\xi(i)}$ and $a_0,a_1$.
Also, for any fixed $\ttp$, for all large enough $x$ and all $t$,
we have
$$\varphi_{\xi(i)}(K^t(x),t)\geq\phi_i(K^t(x))
\geq\phi_i(K(x))>2\,|\ttp|$$
Suppose $W_j$ is infinite. Then $Z_0(i,j,\ttp)$ contains exactly
$2^{2|\ttp|+1}$ elements.
Let $\kmin[]=K_U$ where $U\in\minpr[\words\to\bbbn]$.
Since there are $2^{2|\ttp|+1}-1$ words with length $\leq2\,|\ttp|$,
there is necessarily some element of $x\in Z_0(i,j,\ttp)$ which
is not in $U(\{\ttq:|\ttq|\leq2\,|\ttp|\})$, hence is such that
$\kmin[](x)=K_U(x)>2\,|\ttp|$.
\\
Let $x_0(i,j,\ttp)$ be the largest such element.
It is easy to see that, for all $t$ large enough, we have
$a_0(i,j,\ttp,t)=x_0(i,j,\ttp)$.
Thus, $\alpha_0(i,j,\ttp)=x_0(i,j,\ttp)$ is defined and
$\alpha_0(i,j,\ttp)\in W_j\cap\{\phi_i(K(x))>2\,|\ttp|\}$.
Which proves Point a of the Claim.
\medskip\\
Suppose $\bbbn\setminus W_k$ is infinite and $\psi_m$ is
a total recursive function such that
$\psi_m\geq growth_{\bbbn\setminus W_k}$.
Then there are $2^{2|\ttp|+1}$ elements of $\bbbn\setminus W_k$
which are $\leq\psi_m(2^{2|\ttp|+1})$.
As above, there is necessarily some such element $x$ which is
not in $U(\{\ttq:|\ttq|\leq2\,|\ttp|\})$, hence is such that
$\kmin[](x)=K_U(x)>2\,|\ttp|$.
\\
Let $x_1(i,k,m,\ttp)$ be the largest such element.
It is easy to see that, for all $t$ large enough (namely, for $t$
such that $W_k\cap[0,x_1(i,k,m,\ttp)[\subseteq W_{k,t}$), we have
$a_1(i,j,m,\ttp,t)=x_1(i,j,m,\ttp)$.
Thus, $\alpha_1(i,k,m,\ttp)=x_1(i,k,m,\ttp)$ is defined and
$\alpha_1(i,k,m,\ttp)\in
(\bbbn\setminus W_k)\cap\{\phi_i(K(x))>2\,|\ttp|\}$.
Which proves Point b of the Claim.\hfill{$\Box$ \em (Claim)}
\medskip\\
We conclude the proof of the Lemma as that of Lemma \ref{l:ggKmax}
with analogous points 4,5,6 : the sole modification is to replace
everywhere $\kmax[]$ by $\kmin[]$ and the $\max$ operator by the
$\min$ one.
\end{proof}
%
%
\subsection{Comparing {$\kmin[]$} and {$\kmax[]$} \`a la Barzdins}
\label{ss:KminKmaxOften}

We shall need the following result from \cite{ferbusgrigoKmaxKmin}
(Thm 7.15).

\begin{proposition}\label{p:KKK}
$K\leqct2\,\kmin[]+\kmax[]$.
\end{proposition}
\noindent
Using Prop.\ref{p:KKK}, Lemmas \ref{l:ggKmax}, \ref{l:ggKmin}
yield the following corollary.
\begin{lemma}\label{l:ggKminKmax}$\\ $
{\bf 1.}
Let's denote $\widetilde{\Pi^0_1}$ the family of infinite $\Pi^0_1$
subsets of $\bbbn$ with recursively bounded growth.\\
Suppose $\phi:\bbbn\to\bbbn$ is a total function in $\minpr[]$
which is monotone and tends to $+\infty$.
Then
\begin{enumerate}
\item[i.]
$\{x:\kmax[](x)<\phi(\kmin[](x))\}$ is constructively
$\ (\Sigma^0_1\cup\Pi^0_1,
\exists^{\leq\phi}(\Sigma^0_1\wedge\Pi^0_1))$-dense.
\item[ii.]
$\{x:\kmin[](x)<\phi(\kmax[](x))\}$ is constructively
$\ (\Sigma^0_1\cup\widetilde{\Pi^0_1},
\exists^{\leq\phi}(\Sigma^0_1\wedge\Pi^0_1))$-dense
\end{enumerate}
Moreover, this result is uniform in $\phi$ and, for ii, in a
recursive bound for the $\Pi^0_1$ set, in the sense detailed in
Lemmas \ref{l:ggKmax},\ref{l:ggKmin}.
\medskip\\
{\bf 2.}
Consider Kolmogorov relativizations $K^A,\kmin[A],\kmax[A]$ and
enumerations $(\phi^A_i)_{i\in\bbbn}$ and $(\psi^A_m)_{m\in\bbbn}$ and
$(W^A_i)_{i\in\bbbn}$ of $\minpr[A]$, $\PR[A]$
and $A$-r.e. sets  as in Point 2 of Lemmas \ref{l:ggKmax},
\ref{l:ggKmin}.
\\
Then Point 1 relativizes uniformly in the sense
detailed in Lemmas \ref{l:ggKmax},\ref{l:ggKmin}.
\end{lemma}
\begin{proof}
Let $\phi\in\minpr[]$ be total, monotone increasing and unbounded.
Set
$$\theta(x)
=\min\{\frac{x}{4},\phi(\max(0,\lfloor\frac{x-c}{2}\rfloor))\}$$
Then $\theta$ is also a total, monotone increasing and unbounded
function in $\minpr[]$. Also, one can recursively go from a code
for $\phi$ to one for $\theta$.
Using Lemmas \ref{l:ggKmax}, \ref{l:ggKmin}, it suffices to prove
that, for all $x$,
\begin{eqnarray*}
\kmin[](x)<\theta(K(x))&\Rightarrow&\kmin[](x)<\phi(\kmax[](x))\\
\kmax[](x)<\theta(K(x))&\Rightarrow&\kmax[](x)<\phi(\kmin[](x))
\end{eqnarray*}
We prove the first implication, the second one being similar.
\\
Applying Prop.\ref{p:KKK}, let $c$ be such that, for all $x$,
$$K(x)<2\,\kmin[](x)+\kmax[](x)+c$$
Suppose $\kmin[](x)<\theta(K(x))$.
Then $\kmin[](x)<\frac{1}{4}\,K(x)$, so that
$$K(x)\ <\ 2\,\kmin[](x)+\kmax[](x)+c\
\leq\ \frac{K(x)}{2}+\kmax[](x)+c$$
and $K(x)<2\,(\kmax[](x)+c)$.\\
Therefore,
$\kmin[](x)<\theta(K(x))\leq\theta(2\,(\kmax[](x)+c))
\leq\phi(\kmax[](x))$.
\end{proof}
%
\subsection{Syntactical complexity}
\label{ss:complexity}
%
Whereas $\{x:K(x)<\phi(x)\}$ is r.e. whenever $\phi$ is partial
recursive (cf. Lemma \ref{l:ggK}), the complexity of the sets
considered in Lemmas  \ref{l:ggKmax}, \ref{l:ggKmin},
\ref{l:ggKminKmax} to compare $K, \kmax[], \kmin[]$ is
much higher and does involve bounded quantifications over
boolean combinations of $\Sigma^0_1$ sets as is the case in the
density results obtained in these lemmas.
\begin{proposition}\label{p:complexity}
Let $\phi$ be a total function in $\minpr[]$.
The sets
\medskip\\\medskip\centerline{
$\begin{array}{ccc}
\{x:\kmax[](x)<\phi(K(x))\}
&&\{x:\kmax[](x)<\phi(\kmin[](x))\}\\
\{x:\kmin[](x)<\phi(K(x))\}
&&\{x:\kmin[](x)<\phi(\kmax[](x))\}
\end{array}$}
are all definable by formulas of the form
$$\exists^{\leq\log}\ \forall^{\leq\log}\ (A\wedge B\wedge C)$$
where $A,B,C$ are $\Sigma^0_1\vee\Pi^0_1$.
In particular, theses sets are $\Delta^0_2$
(cf. Prop.\ref{p:boundedDelta02}).
\end{proposition}
\begin{proof}
Without loss of generality, we can suppose that $\kmax[](x)$ and
$\kmin[](x)$ are both $\leq \log(x)$ for all $x$.
Let $U:\bbbn\to\bbbn$ and $V,W,\varphi:\bbbn^2\to\bbbn$ be partial
recursive functions such that $K=K_U$ and $\kmin[]=K_\alpha$ and
$\kmax[]=K_\beta$ and $\phi(x)=\min_t\varphi(x,t)$ where
$\alpha(x)=\min_tV(x,t)$ and $\beta(x)=\min_tW(x,t)$.
\\
Following a usual convention, we shall write
$\exists\ttp^{|\ttp|\leq x}...$ and $\forall\ttp^{|\ttp|\leq x}...$
in place of $\exists\ttp\ (|\ttp|\leq x\ \wedge...)$ and
$\forall\ttp\ (|\ttp|\leq x\ \Rightarrow...)$.
\\
Then $\kmax[](x)<\phi(K(x))$ if and only if
\medskip\\
\indent$\exists{\ttp_1}^{|\ttp_1|\leq\log(x)}\
\exists{\ttp_2}^{|\ttp_2|\leq\log(x)}\
\forall{\ttq_1}^{|\ttq_1|<|\ttp_1|}\
\forall{\ttq_2}^{|\ttq_2|<|\ttp_2|}$
\medskip\\
\indent\indent$[U(\ttp_1)=x\ \wedge\ U(\ttq_1)\neq x$
\medskip\\
\indent\indent$\wedge\ \exists t\ V(\ttp_2,t)=x\
\wedge\ \forall t\ (V(\ttp_2,t)\mbox{ is undefined or }\leq x)$
\medskip\\
\indent\indent
$\wedge\ (\forall t\ (V(\ttq_2,t)\mbox{ is undefined or }\neq x)
\vee\ \exists t\ V(\ttq_2,t)> x)$
\medskip\\
\indent\indent
$\wedge\ \forall t\ (\varphi(|\ttp_1|,t)\mbox{ is undefined or }
|\ttp_2|<\varphi(|\ttp_1|,t))]$
\medskip\\
Which is a formula of the form stated in the Proposition.
All three other cases are similar.
\end{proof}
Bounded quantifications over boolean combinations of $\Sigma^0_1$
sets are also involved for the set of integers with
$K,\kmax[],\kmin[]$ incompressible binary representations.
\begin{proposition}
The set
$$I=\{x:\min(K(x),\kmax[](x),\kmin[](x))\geq\lfloor\log(x)\rfloor-1\}$$
is infinite and is definable by a formula of the form
$$\forall^{\leq\log}\ (A\wedge B)$$
where $A,B$ are $\Sigma^0_1\vee\Pi^0_1$.
In particular, this set is $\Delta^0_2$.
\end{proposition}
\begin{proof}
Without loss of generality we shall suppose that $K\leq\kmax[]$
and $K\leq\kmin[]$.
The usual argument to get incompressible integers works:
there are $\sum_{i<n}2^i=2^n-1$ programs $\ttp$ with length $<n$,
hence at most $2\,(2^n-1)$ integers $x$ such that $\kmax[](x)<n$
or $\kmin[](x)<n$.
Thus, for every $n$, there exists an integer $x\leq 2^{n+1}-1$ such
that $\kmax[](x),\kmin[](x)\geq n$.
Observe that such an $x$ is necessarily in $I$ since
$\log(x)\leq\log(2^{n+1}-1)<n+1$.
Which shows that $I$ is infinite.
\medskip\\
We let $V,W$ be as in the proof of Prop.\ref{p:complexity}.
Then $x\in I$ can be written
\medskip\\
$\forall\ttp^{|\ttp|<\lfloor\log(x)\rfloor-1}\
[(\forall t\ (V(\ttp,t)\mbox{ is undefined or }\neq x)
\vee\ \exists t\ V(\ttp,t)> x)$
\medskip

\hfill{$\wedge\ (\forall t\ (W(\ttp,t)\mbox{ is undefined or }\neq x)
\vee\ \exists t\ W(\ttp,t)< x)]$}
\medskip\\
Which is a formula of the form stated in the Proposition.
All three other cases are similar.
\end{proof}
\begin{remark}
In case $\phi$ is small enough (say $\phi(z)\leq z-1$),
the set $I$ is obviously disjoint from all fours sets considered in
Prop.\ref{p:complexity}.
\end{remark}
%
\subsection{The hierarchy theorem}
\label{ss:hierarchy}
%
We can now prove the central application of the
$\often[]{\uparrow}$ relation and the $\lless[]\,$ and
$\lless[]\uparrow$ orderings. Namely, a strong hierarchy theorem
for $K,\kmax[],\kmin[]$ and their oracular versions using
the successive jumps oracles.
\\
Whereas Thm.\ref{thm:hierarchyInfct} involves the sole $\infct$
ordering, the refinment obtained in Thm.\ref{thm:hierarchy}
below involves a chain of more and more complex orderings which all
refine $\infct$ and are relevant of
Thm.\ref{thm:order} and Cor.\ref{coro:order}.
\begin{theorem}[The hierarchy theorem]\label{thm:hierarchy}
Let $B_n$ be the subclass of $\Delta^0_n$ subsets of $\bbbn$
consisting of sets definable by formulas of the form
$\exists^{\leq\mu}(\Sigma^0_n\wedge\Pi^0_n)$ where
$\mu:\bbbn\to\bbbn$ is a total function which is recursive in
$\emptyset^{(n-1)}$.
Let $\widetilde{\Pi^0_n}$ be the set of $\Pi^0_n$ sets with
$\emptyset^{(n)}$-recursively bounded growth
(cf. Def.\ref{def:recbounded}).\\
Then
\medskip\\
{\bf 1.} $\log
\ggreater[\Sigma^0_1,\Sigma^0_1]{\PR[]}
K
\ggreater[\Sigma^0_1\cup\Pi^0_1,B_1]{\minpr[]\uparrow}
\kmax[]
\ggreater[\Sigma^0_2,\Sigma^0_2]{\PR[\emptyset']}
K^{\emptyset'}...$\medskip

\hfill{$...\ggreater[\Sigma^0_n,\Sigma^0_n]{\PR[\emptyset^{(n-1)}]}
K^{\emptyset^{(n-1)}}
\ggreater[\Sigma^0_n\cup\Pi^0_n,B_n]
{Min_{PR^{\emptyset^{(n-1)}}}\uparrow}
\kmax[\emptyset^{(n-1)}]
\ggreater[\Sigma^0_{n+1},\Sigma^0_{n+1}]{\PR[\emptyset^{(n)}]}
K^{\emptyset^{(n)}}...$}
\medskip\medskip\\
{\bf 2.} $\log
\ggreater[\Sigma^0_1,\Sigma^0_1]{\PR[]}
K
\ggreater[\Sigma^0_1\cup\widetilde{\Pi^0_1},B_1]{\minpr[]\uparrow}
\kmin[]
\ggreater[\Sigma^0_2,\Sigma^0_2]{\PR[\emptyset']}
K^{\emptyset'}...$\medskip

\hfill{$...\ggreater[\Sigma^0_n,\Sigma^0_n]{\PR[\emptyset^{(n-1)}]}
K^{\emptyset^{(n-1)}}
\ggreater[\Sigma^0_n\cup\widetilde{\Pi^0_n},B_n]
{Min_{PR^{\emptyset^{(n-1)}}}\uparrow}
\kmin[\emptyset^{(n-1)}]
\ggreater[\Sigma^0_{n+1},\Sigma^0_{n+1}]{\PR[\emptyset^{(n)}]}
K^{\emptyset^{(n)}}...$}
\medskip\medskip\\
{\bf 3.}
There is a constant $c$ such that all $\supct$ inequalities in
1 and 2 (which are inherent to the $\ggreater[]\,$ and
$\ggreater[]\uparrow$ orderings) are $>$ inequalities up to $c$.
\medskip\\
{\bf 4.} Though $\kmax[]$ and $\kmin[]$ are $\leqct$ incomparable,
we have
\begin{eqnarray*}
\kmax[\emptyset^{(n-1)}]
&\often[\Sigma^0_n\cup\Pi^0_n,B_n]
{Min_{PR^{\emptyset^{(n-1)}}}\uparrow}&
\kmin[\emptyset^{(n-1)}]
\\
\kmin[\emptyset^{(n-1)}]
&\often[\Sigma^0_n\cup\Pi^{0,\leq rec^{\emptyset^{(n-1)}}}_n,B_n]
{Min_{PR^{\emptyset^{(n-1)}}}\uparrow}
&\kmax[\emptyset^{(n-1)}]
\end{eqnarray*}
\end{theorem}
\begin{proof}
{\em 1. $\leqct$ inequalities.}
Inequality $\log\geqct K$ is well-known.
The inclusions (cf. Prop.\ref{p:MinPRsubsetPR0'})
\\\centerline{$\PR[\emptyset^{(n)}]\subseteq\minpr[\emptyset^{(n)}]
\subseteq\PR[\emptyset^{(n+1)}]\ \ ,\ \ 
\PR[\emptyset^{(n)}]\subseteq\maxpr[\emptyset^{(n)}]
\subseteq\PR[\emptyset^{(n+1)}]$}
yield inequalities
\\\medskip\centerline{$K^{\emptyset^{(n)}}
\geqct \kmin[\emptyset^{(n)}]\geqct K^{\emptyset^{(n+1)}}\ \ ,\ \ 
K^{\emptyset^{(n)}}\geqct \kmax[\emptyset^{(n)}]
\geqct K^{\emptyset^{(n+1)}}$.}
{\em 2. Inequalities $...\ggreater[]\,K^{\emptyset^{(i)}}$.}
Lemma \ref{l:ggK} with $A=\emptyset$ and
$\varphi\circ\log$ in place of $\varphi$ yields inequality
$\ \log\ \ggreater[\Sigma^0_1,\Sigma^0_1]{\PR[]} K$.
\\
Since $\kmax[A]$ is recursive in $A'$
(cf. Thm.\ref{thm:KminRecIn0'}),
Lemma \ref{l:ggK} with $A=\emptyset^{(n-1)}$ and
$\varphi\circ\kmax[\emptyset^{(n-1)}]\ $ in place of $\varphi$
yields inequality $\kmax[\emptyset^{(n-1)}]\
\ggreater[\Sigma^0_{n+1},\Sigma^0_{n+1}]{\PR[\emptyset^{(n)}]}
\ K^{\emptyset^{(n)}}$.
Idem with $\kmin[]$.
\medskip\\
{\em 3. Inequalities $K^{\emptyset^{(i)}}\ggreater[]\,...$.}
Direct application of Lemmas \ref{l:ggKmax}, \ref{l:ggKmin}.
\medskip\\
{\em 4. Point 3 of the theorem.}
This is the benefit of the uniform oracular property
obtained in Lemmas \ref{l:ggK}, \ref{l:ggKmax}, \ref{l:ggKmin},
\ref{l:ggKminKmax}.
\medskip\\
{\em 5.} Finally, the $\often[]\,$ relations
(Point 4 of the theorem) are direct application of Lemma
\ref{l:ggKminKmax}.
\end{proof}
\begin{remark}
The scattered character of comparisons with respect to the
$\lless[]\,$ orderings is unavoidable since all complexities
$K,\kmax[],\kmin[],...,K^{\emptyset^{(n)}}$ are equal up to
a constant on the infinite set of integers with
$K^{\emptyset^{(n)}}$ incompressible binary representations.
\end{remark}

\end{document}